\documentclass[reqno]{article}
\usepackage{arxiv}

\usepackage[utf8]{inputenc}         
\usepackage[T1]{fontenc}            
\usepackage{lmodern}
\usepackage[dvipsnames]{xcolor}

\usepackage{amsmath}
\usepackage{amssymb,amsthm}
\usepackage{thmtools}
\usepackage{url}                    
\usepackage{tabu}                   
\usepackage{bbm}                    
\usepackage[mathscr]{eucal}         
\usepackage{mathrsfs}               
\usepackage{upgreek}                
\usepackage{colonequals}            
\usepackage{mathtools}              
\usepackage{caption}
\usepackage{subcaption}
\usepackage{enumitem}

\usepackage[colorlinks=true, allcolors=blue, linktoc=page]{hyperref}

\usepackage[style=alphabetic]{biblatex}
\usepackage[capitalize,noabbrev]{cleveref}              

\usepackage{graphicx}
\usepackage{tikz}

\usetikzlibrary{patterns}
\usetikzlibrary{arrows}
\usetikzlibrary{cd}
\usetikzlibrary{decorations.pathreplacing,calligraphy}
\tikzset{>=stealth}

\theoremstyle{plain}
\newtheorem{thm}{Theorem}[section]
\newtheorem{lem}[thm]{Lemma}
\newtheorem{prop}[thm]{Proposition}
\newtheorem{cor}[thm]{Corollary}
\newtheorem{introthm}{Theorem}

\theoremstyle{definition}
\newtheorem{defn}[thm]{Definition}

\theoremstyle{remark}
\newtheorem{rem}[thm]{Remark}

\crefformat{chapter}{\S#2#1#3}
\crefmultiformat{chapter}{\S\S#2#1#3}{ and~#2#1#3}{, #2#1#3}{, and~#2#1#3}

\crefformat{section}{\S#2#1#3}
\crefmultiformat{section}{\S\S#2#1#3}{ and~#2#1#3}{, #2#1#3}{, and~#2#1#3}

\crefformat{subsection}{\S#2#1#3}
\crefmultiformat{subsection}{\S\S#2#1#3}{ and~#2#1#3}{, #2#1#3}{, and~#2#1#3}

\crefformat{subsubsection}{\S#2#1#3}
\crefmultiformat{subsubsection}{\S\S#2#1#3}{ and~#2#1#3}{, #2#1#3}{, and~#2#1#3}

\crefname{equation}{}{}

\newcommand{\todo}[1]{\textbf{\color{blue} #1}}

\def\textclap#1{\hbox to 0pt{\hss#1\hss}}

\newcommand{\cA}{\mathcal{A}}

\newcommand{\cE}{\mathcal{E}}

\newcommand{\cH}{\mathcal{H}}

\newcommand{\cL}{\mathcal{L}}
\newcommand{\cM}{\mathcal{M}}

\newcommand{\cO}{\mathcal{O}}

\newcommand{\cQ}{\mathcal{Q}}
\newcommand{\cR}{\mathcal{R}}
\newcommand{\cS}{\mathcal{S}}

\newcommand{\cX}{\mathcal{X}}
\newcommand{\cY}{\mathcal{Y}}

\newcommand{\bB}{\mathbf{B}}

\newcommand{\bbA}{\mathbb{A}}

\newcommand{\bbN}{\mathbb{N}}

\newcommand{\bbQ}{\mathbb{Q}}

\newcommand{\bbZ}{\mathbb{Z}}

\newcommand{\C}{\mathbb{C}}
\newcommand{\N}{\mathbb{N}}

\newcommand{\id}{\mathrm{id}}

\newcommand{\Mat}{\mathsf{Mat}}
\newcommand{\GL}{\mathsf{GL}}
\newcommand{\Gr}{\mathsf{Gr}}

\newcommand{\eps}{\upepsilon}
\newcommand{\del}{\updelta}

\newcommand{\rH}{\mathrm{H}}
\newcommand{\BM}{\mathrm{BM}}

\newcommand{\pt}{\mathrm{pt}}

\DeclareMathOperator{\rk}{\mathsf{rk}}
\let\ker\relax
\DeclareMathOperator{\ker}{\mathsf{ker}}
\let\im\relax
\DeclareMathOperator{\im}{\mathsf{im}}

\DeclareMathOperator{\gra}{\mathsf{gr}}

\newcommand{\UU}{\mathsf{U}}
\newcommand{\LL}{\mathsf{L}}

\newcommand{\TT}{\mathsf{T}}
\newcommand{\TU}{\mathsf{TU}}

\DeclareFontFamily{U}{mathb}{}
\DeclareFontShape{U}{mathb}{m}{n}{
  <-5.5> mathb5
  <5.5-6.5> mathb6
  <6.5-7.5> mathb7
  <7.5-8.5> mathb8
  <8.5-9.5> mathb9
  <9.5-11.5> mathb10
  <11.5-> mathbb12
}{}
\DeclareRobustCommand{\bsquare}{%
        \text{\usefont{U}{mathb}{m}{n}\symbol{"0D}}%
}

\renewcommand{\d}{\mathsf{d}}

\newcommand{\stab}{\mathsf{stab}}
\newcommand{\forg}{\mathsf{forg}}

\newcommand{\Ch}{\mathsf{Ch}}

\input{xdasharrows}

\title[Categorifying Quiver Linking/Unlinking using CoHA Modules]{Categorifying Quiver Linking/Unlinking using CoHA Modules}
\author[O. van Garderen]{Okke van Garderen}

\begin{document}

\maketitle
\begin{abstract}
  The knots-quivers correspondence is a relation between knot invariants and enumerative invariants of quivers,
  which in particular translates the knot operations of \emph{linking} and \emph{unlinking} to certain operations on quivers.
  In this paper we show that the moduli spaces of a quiver and its linking/unlinking are naturally related, giving a purely representation-theoretic interpretation of these operations.
  We obtain a relation between the cohomologies of these spaces which is moreover compatible with a natural action of the Cohomological Hall Algebra.
  The result is a categorification of quiver linking/unlinking at the level of CoHA modules.
\end{abstract}

\tableofcontents

\section{Introduction}

Generating series are a well-known tool for comparing mathematical structures, and can serve as a first approximation for deeper relations between different subjects.
One important example is the \emph{knots-quivers correspondence} of \cite{KRSS19}, which conjectures a relation between the generating series of knots, which encode their HOMFLY-PT polynomials, and the generating series of corresponding quivers, given by the $q$-refined DT generating function
\begin{equation}\label{eq:DTgen}
  \bbA_Q(x,q) = \sum_{\del\in\N Q_0} \sum_{n\in\bbZ} \frac{(-q^{1/2})^{\chi_Q(\del,\del)}}{\prod_{i\in Q_0} (1-q)(1-q^2)\cdots(1-q^{\del_i})} \cdot x^\del,
\end{equation}
counting moduli of representations of a quiver $Q$.
One piece of evidence that the knots-quivers correspondence is part of a deeper relation is provided in \cite{EKL20}.
Motivated by physics, the authors construct an analogue of the knot operations of \emph{linking} and \emph{unlinking} for quivers:
\[
  Q \leadsto Q^\LL,\quad\quad Q\leadsto Q^\UU,
\]
and show that the generating series $\bbA_Q(x,q)$, $\bbA_{Q^\LL}(x,q)$, and $\bbA_{Q^\UU}(x,q)$ are related in the same way as the generating series for the corresponding knots.

Although the appearance of the quivers $Q^\LL$ and $Q^\UU$ in \cite{EKL20} is well-motivated from a physical standpoint, their meaning is somewhat mysterious when viewed purely from the quiver side:
the construction involves adding a new node, as well as adding/removing several arrows, in a way which does not seem to correspond to known operations on quivers in the literature.
This raises the question: is there a way to express the linking/unlinking using more standard representation theoretic tools?

Secondly there is the issue of \emph{categorification}: replacing enumerative invariants by more refined structures.
It is well-known that \Cref{eq:DTgen} is the Poincaré series of the \emph{Cohomological Hall Algebra} (CoHA)
\[
  \cH_Q \colonequals \bigoplus_{\del\in\N Q_0} \rH^\bullet(\cM_\del(Q),\bbQ)[-\chi_Q(\del,\del)],
\]
where $\cM_\del(Q)$ is the moduli stack of representations of dimension $\del$ and $\chi_Q$ is the Euler pairing.
Another question is then if linking/unlinking yields a relation between $\cH_Q$, $\cH_{Q^\LL}$, and $\cH_{Q^\UU}$.

The goal of this paper is to answer the above two questions.
We find a geometric relation between the moduli spaces of $Q$ and $Q^\LL$/$Q^\UU$ using standard representation-theoretic tools.
In both cases we show that this leads to a module action of $\cH_Q$ on the CoHAs $\cH_{Q^\LL}$ and $\cH_{Q^\UU}$, which we interpret as a categorification of the relations between the generating series.

\subsection{Unlinking}

Given any finite quiver $Q$ with a two-cycle consisting of two arrows $c\colon 0\to 1$ and $d \colon 1\to 0$ between distinguished vertices $0,1$, we can consider the \emph{unlinked} quiver $Q^\UU$.
This is obtained by following the recipe of \cite{EKL20} and consists of removing $c$ and $d$, adding a new vertex $\star \in Q_1^\UU$ and a number of arrows starting/ending in $\star$.
A simple example is illustrated below:
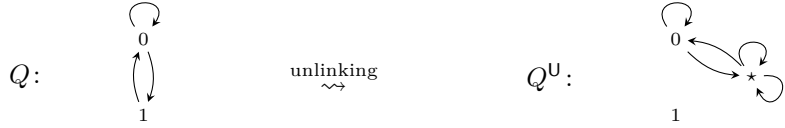
\[
  \begin{tikzpicture}
    \begin{scope}[yshift=3.5cm]
      \node[outer sep=1pt,inner sep=1.5pt] (a) at (0,.5) {$\scriptstyle 0$};
      \node[outer sep=1pt,inner sep=1.5pt] (b) at (0,-.5) {$\scriptstyle 1$};
      \node (d) at (-1.5,0) {$Q\colon$};
      \draw[->] (a) to[bend left=20] (b);
      \draw[->] (b) to[bend left=20] (a);
      \draw[->] (a) to[out=125,in=55,min distance=5mm] (a);
      \node at (2.5,0) {$\overset{\text{unlinking}}\leadsto$};
    \end{scope}      
    \begin{scope}[yshift=3.5cm,xshift=7cm]          
      \node[outer sep=1pt,inner sep=1.5pt] (a) at (0,.5) {$\scriptstyle 0$};
      \node[outer sep=1pt,inner sep=1.5pt] (b) at (0,-.5) {$\scriptstyle 1$};
      \node[outer sep=1pt,inner sep=1.5pt] (e) at (1,0) {$\scriptstyle \star$};
      \node (d) at (-1.6,0) {$Q^\UU\colon$};
      \draw[->] (a) to[bend right=20] (e);
      \draw[->] (e) to[bend right=20] (a);
      \draw[->] (a) to[out=125,in=55,min distance=5mm] (a);
      \draw[->] (e) to[out=125,in=55,min distance=5mm] (e);
      \draw[->] (e) to[out=5,in=-65,min distance=5mm] (e);
    \end{scope}
  \end{tikzpicture}
\]
For each dimension vector $\eps\in \N Q_0^\UU$ there is a moduli space $\cM_\eps(Q^\UU)$ of $Q^\UU$-representations modulo the action of a symmetry group $\GL_\eps$, whose cohomologies encode the generating function of $Q^\UU$.

Instead of constructing $Q^\UU$ one can consider a \emph{rank stratification} of the moduli spaces $\cM_\del(Q)$ with locally closed strata
depending on the rank of the action  $\rho_d \colon \C^{\del_1} \to \C^{\del_0}$ of $d$:
\[
  \cS_{\del,\ell} = \big\{\ \rho \in \cM_\del(Q) \ \mid\ \rk \rho_d = \ell\ \big\}.
\]
For each such representation $\rho_d$ identifies an $\ell$-dimensional subspace in $\C^{\del_1}$ with one in $\C^{\del_0}$,
and the stratum $\cS_{\del,\ell}$ can therefore be represented as the set of representations with 
\begin{equation}\label{eq:dblock}
  \rho_d = \left(\begin{array}{c|c}
    0 & 0\\\hline
    0\vphantom{\sum^{I}} & I_{\ell\times\ell}
  \end{array}\right),
\end{equation}
modulo the action of a subgroup $P_{\del,\ell}$ of the full symmetry group $\GL_\del$ of the quiver.
In \Cref{sec:unlink} we show that for each pair $(\del,\ell)$ the space of block matrices is homotopic to the space of representations of $Q^\UU$ for some unique dimension vector $\eps$, and $\GL_\eps$ appears naturally as the Levi subgroup of $P_{\del,\ell}$.

\begin{introthm}[{\Cref{prop:unlinkstrat},~\Cref{thm:gradeddecomp}}]\label{introthm:A}
  Let $Q$ be any quiver admitting an unlinking $Q^\UU$.
  Then for each $\eps\in \N Q_0^\UU$ there is a dimension vector $\del = u(\eps)$ with a homotopy equivalence
  \begin{equation}\label{eq:unltostrat}
    \cM_\eps(Q^\UU) \to \cS_{(u(\eps),\eps_\star)},
  \end{equation}
  and for each $\del \in \N Q_0$ this yields a decomposition in cohomology
  \begin{equation}\label{eq:decompmentioned}
    \rH^\bullet(\cM_\del(Q),\bbQ)[-\chi_Q(\del,\del)] \cong \bigoplus_{u(\eps) = \del} \rH^\bullet(\cM_\eps(Q^\UU),\bbQ)[-\chi_{Q^\LL}(\eps,\eps)].
  \end{equation}
\end{introthm}

Since the Poincaré series of the shifted cohomologies are the coefficients of the generating series of $Q$ and $Q^\UU$ this theorem recovers the relation $\bbA_{Q^\UU}(x,q)|_{x_\star = x_0x_1} = \bbA_Q(x,q)$ found in \cite{EKL20}.

\Cref{introthm:A} categorifies the relation between $\bbA_{Q^\UU}(x,q)$ and $\bbA_Q(x,q)$ at the level of vector spaces.
We further show that the decomposition can be made compatible with the module action of the CoHA $\cH_Q$, which yields the following categorification at the level of CoHA modules.

\begin{introthm}[{\Cref{thm:filtbyideals}}]\label{introthm:B}
  Then for each choice of unlinking $Q^\UU$,
  the CoHA $\cH_Q$ has a descending filtration $\cH_Q = \cR_0 \supset \cR_1 \supset \ldots$ by graded right ideals,
  such that the decomposition \Cref{eq:decompmentioned} induces an isomorphism of $\bbZ\times\N Q_0$-graded vector spaces
  \[
    \gra_\bullet \cH_Q = \bigoplus_{p\in\N} \cR_p/\cR_{p+1} \cong \cH_{Q^\UU},
  \]
  which makes $\cH_{Q^\UU}$ into an graded $\cH_Q$-module.
\end{introthm}

\subsection{Linking}

Given again a finite quiver $Q$ and a choice of distinct nodes $0,1\in Q_0$ there is a \emph{linked} quiver $Q^\LL$, which consists of a new two-cycle between $0$ and $1$, a new vertex denoted $\bsquare\in Q^\LL_0$, as well as additional arrows starting/ending in $\bsquare$.
A simple example is given below:
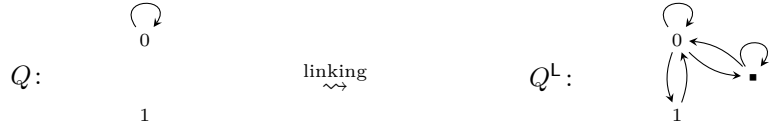
\[
  \begin{tikzpicture}
    \begin{scope}[yshift=3.5cm]
      \node[outer sep=1pt,inner sep=1.5pt] (a) at (0,.5) {$\scriptstyle 0$};
      \node[outer sep=1pt,inner sep=1.5pt] (b) at (0,-.5) {$\scriptstyle 1$};
      \node (d) at (-1.5,0) {$Q\colon$};
      \draw[->] (a) to[out=125,in=55,min distance=5mm] (a);
      \node at (2.5,0) {$\overset{\text{linking}}\leadsto$};
    \end{scope}      
    \begin{scope}[yshift=3.5cm,xshift=7cm]          
      \node[outer sep=1pt,inner sep=1.5pt] (a) at (0,.5) {$\scriptstyle 0$};
      \node[outer sep=1pt,inner sep=1.5pt] (b) at (0,-.5) {$\scriptstyle 1$};
      \node[outer sep=1pt,inner sep=1.5pt] (e) at (1,0) {$\scriptstyle \bsquare$};
      \node (d) at (-1.6,0) {$Q^\LL\colon$};
      \draw[->] (a) to[bend right=20] (b);
      \draw[->] (b) to[bend right=20] (a);      
      \draw[->] (a) to[bend right=20] (e);
      \draw[->] (e) to[bend right=20] (a);
      \draw[->] (a) to[out=125,in=55,min distance=5mm] (a);
      \draw[->] (e) to[out=125,in=55,min distance=5mm] (e);
    \end{scope}
  \end{tikzpicture}
\]
Again, we wish to interpret the moduli space $\cM_\gamma(Q^\LL)$ for dimension vectors $\gamma \in \N Q_0^\LL$ in relation to those of $Q$.
In the linking setup this requires another quiver: the two-cycle quiver $Q^{\TT}$ obtained by simply adding a two-cycle; in the example above this yields:
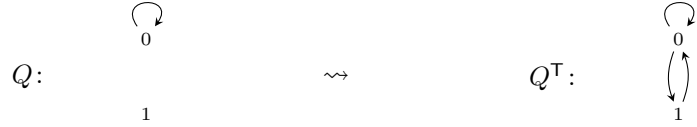
\[
  \begin{tikzpicture}
    \begin{scope}[yshift=3.5cm]
      \node[outer sep=1pt,inner sep=1.5pt] (a) at (0,.5) {$\scriptstyle 0$};
      \node[outer sep=1pt,inner sep=1.5pt] (b) at (0,-.5) {$\scriptstyle 1$};
      \node (d) at (-1.5,0) {$Q\colon$};
      \draw[->] (a) to[out=125,in=55,min distance=5mm] (a);
      \node at (2.5,0) {$\leadsto$};
    \end{scope}      
    \begin{scope}[yshift=3.5cm,xshift=7cm]          
      \node[outer sep=1pt,inner sep=1.5pt] (a) at (0,.5) {$\scriptstyle 0$};
      \node[outer sep=1pt,inner sep=1.5pt] (b) at (0,-.5) {$\scriptstyle 1$};
      \node (d) at (-1.6,0) {$Q^\TT\colon$};
      \draw[->] (a) to[bend right=20] (b);
      \draw[->] (b) to[bend right=20] (a);      
      \draw[->] (a) to[out=125,in=55,min distance=5mm] (a);
    \end{scope}
  \end{tikzpicture}
\]
We consider moduli spaces $\cX_\del^{(k)}$ of \emph{stably framed representations}: pairs $(\rho,f)$ of a $Q^\TT$-representation with a framing data $f\colon \C^{\del_0} \to \C^k$ satisfying a certain stability condition with respect to the added arrows.
We show that for every $\gamma = (\gamma',\gamma_\bsquare) \in \N Q_0^\LL$ there is a unique pair $(\del,k) \in\N Q_0 \times \N$ with a homotopy equivalence to the quotient stack
\begin{equation}\label{eq:linkheq}
  \cM_\gamma(Q^\LL) \longrightarrow \overline\cX_{\del}^{(k)} \colonequals \left[\cX_{\del}^{(k)}/\GL_k\right],
\end{equation}
where $\GL_k$ acts from the left on the framing data.
These maps are again constructed by considering a certain block form, which is acted upon by a subgroup of $\GL_\del$ with Levi subgroup $\GL_\gamma$.

To categorify the relations between the generating series we consider the shifted cohomologies
\[
  \overline\cH_{Q^\TT}^{(k)} \colonequals {\bigoplus_{\del\in\N Q_0^\TT}} \rH^\bullet(\overline\cX_\del^{(k)},\bbQ)[-\chi_{Q^\TT}(\del,\del)-k^2],
\]
and we show that their sum is isomorphic to $\cH_{Q^\LL}$ via the homotopy equivalences \Cref{eq:linkheq}.
Moreover, this isomorphism is graded after shifting $\overline\cH_{Q^\TT}^{(k)}$ and \emph{regrading} $\cH_{Q^\LL}$ over $\bbZ \times \N Q_0$.
This regrading coincides with a change-of-variables relating the generating series $\bbA_Q(x,q)$ and $\bbA_{Q^\LL}(x,q)$ in \cite{EKL20}.
We categorify further by adding a chain differential between these modules, as follows.

\begin{introthm}[\Cref{thm:differentialconstruction}]\label{introthm:C}
  Let $Q^\TT$ be a choice of two-cycle quiver obtained from a quiver $Q$.
  The spaces $\overline\cH_{Q^\TT}^{(k)}[k]$ fit into a chain complex of $\bbZ\times\N Q_0$-graded vector spaces $(\overline\cH_{Q^\TT},\d)$ of the form
  \[
    \ldots \leftarrow 0 \leftarrow \overline \cH_{Q^\TT}^{(0)} \xleftarrow{\ \d_0\ }  \overline \cH_{Q^\TT}^{(1)}[1] \xleftarrow{\ \d_1\ } \ldots \quad \in \Ch_\bullet^{\bbZ\times \N Q_0}(\bbQ),
  \]
  with homology given by $\cH_Q$ concentrated in homological degree $0$.
  In particular, this endows the CoHA of $Q^\LL$ with the structure of a chain complex $(\cH_{Q^\LL},\d) \cong \left(\overline \cH_{Q^\TT},\d\right)$ quasi-isomorphic to $\cH_Q$.
\end{introthm}

We construct the differentials \Cref{introthm:C} by considering the unlinking $Q^\TU$ of $Q^\TT$ at the added two-cycle.
There are again spaces of framed representations $\cY_\eps^{(k)}$ for $Q^\TU$ and the shifted cohomologies 
\[
  \overline\cH_{Q^\TU}^{(k)} = {\bigoplus_{\eps\in \N Q_0^\TU}} \rH^\bullet\left(\overline\cY_\eps^{(k)},\bbQ\right)[-\chi_{Q^\TU}(\eps,\eps)-k^2]
\]
of the $\GL_k$-quotients $\smash{\overline\cY_\eps^{(k)}}$ are isomorphic to $\smash{\overline\cH_{Q^\TT}}$ via a decomposition similar to \Cref{introthm:A}.
The differential $\d$ is then defined between the spaces $\overline\cH_{Q^\TU}^{(k)}[k]$.

In fact, we can categorify further using a natural $\cH_Q$-module structure on the spaces $\overline\cH_{Q^\TU}^{(k)}$,
which is closely related to the CoHA-module structure on framed moduli considered in \cite{S16,FR18}.
This module structure is compatible with the differential, yielding the following theorem, which is stated only for symmetric quivers for simplicity.

\begin{introthm}[{\Cref{thm:DGmodule}}]
  Assume $Q$ is symmetric.
  The differential $\d$ is $\cH_Q$-linear for a natural $\cH_Q$-module structure on $\overline\cH_{Q^\TU}^{(k)}$.
  Hence $(\cH_{Q^\LL},\d) \cong (\overline\cH_{Q^\TT},\d) \cong (\overline \cH_{Q^\TU},\d)$ are chain complexes of $\bbZ\times\N Q_0$-graded $\cH_Q$-modules, which are quasi-isomorphic to the free module $\cH_Q$.
\end{introthm}

The symmetric assumption can be dropped,
at the cost of losing the $\bbZ$-grading.

\subsection{Further questions}
We firstly remark that the $\cH_Q$-module structure we find on $\cH_{Q^\LL}$ and $\cH_{Q^\UU}$, does not use the CoHA product that already exists on these algebras.
It would be interesting to explore the interaction of the algebra structures with the $\cH_Q$-module structure.

Secondly, there is the question of potentials.
In the framework \cite{KS11} the CoHA $\cH_Q$ is the special case of the ``critical'' CoHA $\cH_{Q,W}$ of a quiver with potential $(Q,W)$ with $W=0$.
Since our setup is purely geometrical, it might be possible to extend linking and unlinking to quivers with potential using the comprehensive sheaf-theoretic framework that exists in the literature (as in e.g. \cite{DM20}).

Finally, we want to mention that another categorification of (un)linking was presented in \cite{DFKR24}
using a certain supercommutative algebra $\cA_Q$ defined in \cite{DFR22}.
The authors show that $\cA_Q$ is related to $\cA_{Q^\LL}$ by a filtration and to $\cA_{Q^\UU}$ by a chain resolution, whereas we find the opposite.
It was already shown in \cite{DFR22} that the Koszul dual of $\cA_Q$ is closely related to DT theory, and it would therefore be interesting to explore if there is a duality relating the two categorifications.


\subsection*{Acknowledgements}
  The author thanks Vladimir Dotsenko for interesting discussions that led him to write this paper.

\section{Preliminaries}

\subsection{Quiver moduli}

Let $Q = (Q_0,Q_1)$ be a finite quiver, where $Q_0$ denotes the vertex set and $Q_1$ the set of arrows.
We denote arrows by $a\colon i\to j$ where $i$ and $j$ are the source and target.
We write $\N Q_0$ for the monoid of dimension vectors, for $\del,\del'\in\N Q_0$ we write
\[
  \chi_Q(\del,\del') \colonequals \sum_{i\in Q_0} \del_i\del_i' - \sum_{a\colon\! i\to j\kern 1pt \in Q_1} \del_j\del_i',
\]
for the \emph{Euler pairing}.
The quiver $Q$ is said to be \emph{symmetric} if $\chi_Q$ is a symmetric bilinear form.

For any fixed $\del\in \N Q_0$ we consider the affine space of representations
\[
  R_\del(Q) = \prod_{a\colon\! i\to j\kern 1pt \in Q_1} \Mat_\C(\del_j,\del_i),
\]
whose elements we write as tuples $\uprho = (\uprho_a)_{a\in Q_1}$.
Isomorphism classes of representations correspond to the orbits of the algebraic group
\[
  \GL_\del \colonequals \prod_{i\in Q_0} \GL_i(\C),
\]
which acts on $R_\del(Q)$ by base-change.
The associated \emph{quiver moduli space} is the quotient stack
\[
  \cM_\del(Q) = [R_\del(Q)/\GL_\del].
\]
We note that $\cM_\del(Q)$ is a smooth Artin stack of dimension $\dim \cM_\del(Q) = -\chi_Q(\del,\del)$.

\subsection{Cohomologies and CoHA}

Because quiver moduli spaces are quotient stacks, their singular cohomology is described naturally via equivariant cohomology\footnote{Here we always take cohomology with coefficients in $\bbQ$, abbreviating $\rH^\bullet(-) = \rH^\bullet(-,\bbQ)$.}:
\[
  \rH^\bullet(\cM_\del(Q)) = \rH^\bullet_{\GL_\del}(R_\del(Q)) \cong \rH^\bullet_{\GL_\del}(\pt),
\]
where the second isomorphism follows because $R_\del(Q)$ is contractible.
Since $\cM_\del(Q)$ is smooth there is also a dual description using (equivariant) Borel--Moore homology
\[
  \rH^\bullet(\cM_\del(Q)) \cong \rH_{-\bullet-2\chi_Q(\del,\del)}^\BM(\cM_\del(Q)) \cong \rH_{-\bullet-2\chi_Q(\del,\del)}^{\BM,\GL_\del}(R_\del(Q)),
\]
where the shift is twice the dimension of $\cM_{\del}(Q)$.
We adopt the standard convention of normalising the cohomology by a shift in dimension, and consider the $\bbZ$-graded vector spaces
\[
  \cH_{Q,\del}^\bullet \colonequals \rH^{\bullet - \chi_Q(\del,\del)}(\cM_\del(Q)) \cong \rH_{-\bullet - \chi_Q(\del,\del)}^\BM(\cM_\del(Q)).
\]
The cohomology of quiver moduli can be given the structure of a graded algebra, the Cohomological Hall Algebra (CoHA), which was defined by Kontsevich--Soibelman \cite{KS11} as an analogue in characteristic $0$ of the Ringel--Hall algebra over a finite field.
The underlying space is the direct sum 
\[
  \cH_Q \colonequals \bigoplus_{\del\in \N Q_0} \cH_{Q,\del},
\]
which is $\bbZ\times\N Q_0$-graded: there is a $\N Q_0$-grading by $\del$ and a cohomological grading on each $\cH_{Q,\del}$.
The algebra structure is defined via an extension product, which we recall below.

Given $\del^{(1)}, \del^{(2)}\in \N Q_0$ we let $R_{\del^{(1)},\del^{(2)}}(Q) \subset R_{\del^{(1)} + \del^{(2)}}(Q)$ denote the closed subvariety of representations $\rho = (\rho_a)_{a\in Q_1}$ which are block upper triangular,
in the sense that for every arrow $a\colon i\to j$ the map $\rho_a$ has the form
\[
  \rho_a  =
  \begin{pmatrix}
    \rho_a^{(1)} & * \\
    0 & \rho_a^{(2)}
  \end{pmatrix},
\]
for some $\rho_a^{(n)} \in \Mat_\C(\del^{(n)}_j,\del^{(n)}_i)$.
With this notation, the sequence $\rho^{(1)} = (\rho_a^{(1)})_{a\in Q_1}$ defines a subrepresentation of $\rho$, while the maps $\rho^{(2)} = (\rho_a^{(2)})_{a\in Q_1}$ define a quotient.
We therefore have a natural correspondence
\[
  R_{\del^{(1)}}(Q) \times R_{\del^{(2)}}(Q) \xleftarrow{\ p\ } E_{\del^{(1)},\del^{(2)}}(Q) \xhookrightarrow{\ q\ } R_{\del^{(1)}+\del^{(2)}}(Q),
\]
where $p$ projects onto the pair $(\rho^{(1)},\rho^{(2)})$ and $q$ is the inclusion.
The fibre $p^{-1}(\rho^{(1)},\rho^{(2)})$ over a fixed pair parametrises the extensions $0\to \rho^{(1)} \to \rho \to \rho^{(3)} \to 0$ between $\rho^{(1)}$ and $\rho^{(2)}$.
Two such extensions are equivalent precisely if they lie in the same orbit for the parabolic subgroup $\GL_{\del^{(1)},\del^{(2)}} \subset \GL_{\del^{(1)} + \del^{(2)}}$ consisting of elements $g$ for which each $g_i$ is block upper triangular as above.
Taking the quotient by this action therefore yields the \emph{moduli stack of extensions}
\[
  \cE_{\del^{(1)},\del^{(2)}}(Q) = \left[E_{\del^{(1)},\del^{(2)}}(Q)/\GL_{\del^{(1)},\del^{(2)}}\right],
\]
between representations of dimension $\del^{(1)}$ and $\del^{(2)}$.
One checks that the maps $p$ and $q$ are equivariant, and therefore descend to a diagram of moduli stacks
\begin{equation}\label{eq:COHAcorr}
  \cM_{\del^{(1)}}(Q) \times \cM_{\del^{(2)}}(Q) \xleftarrow{\ p\ } \cE_{\del^{(1)},\del^{(2)}}(Q) \xrightarrow{\ q\ } \cM_{\del^{(1)}+\del^{(2)}}(Q).
\end{equation}
The map $q$ is a proper morphism of relative dimension $\chi_Q(\del^{(2)},\del^{(1)})$,
so there is an induced map on cohomology of the form
\[
  \rH^{\bullet}(\cM_{\del^{(1)}}(Q) \times \cM_{\del^{(2)}}(Q)) \xrightarrow{\ p^*\ }
  \rH^{\bullet}(\cE_{\del^{(1)},\del^{(2)}}(Q))
  \xrightarrow{\ q_*\ }
  \rH^{\bullet - 2\chi_Q(\del^{(2)},\del^{(1)})}(\cM_{\del^{(1)}+\del^{(2)}}(Q)).
\]
Writing $\pi_i \colon \cM_{\del^{(1)}}(Q) \times \cM_{\del^{(2)}}(Q) \to \cM_{\del^{(i)}}(Q)$ for the projections, we have a K\"unneth isomorphism
\[
  \rH^{\bullet}(\cM_{\del^{(1)}}(Q)) \otimes \rH^\bullet(\cM_{\del^{(2)}}(Q)) \to \rH^{\bullet}(\cM_{\del^{(1)}}(Q) \times \cM_{\del^{(2)}}(Q))\quad
  \beta\otimes\alpha \mapsto \pi_1^*\beta \cup \pi_2^*\alpha,
\]
where $\cup$ denotes the cup-product in cohomology.
Composing this with $q_*p^*$ and taking into account the shifts in $\cH_Q$ yields the extension product
\[
  \cdot \colon \cH_{Q,\del^{(1)}} \otimes \cH_{Q,\del^{(2)}} \to \cH_{Q,\del^{(1)} + \del^{(2)}} [\chi_Q(\del^{(1)},\del^{(2)}) - \chi_Q(\del^{(2)},\del^{(1)})],
\]
which makes $\cH_Q$ into an $\N Q_0$-graded algebra.
If $Q$ is a symmetric quiver then the product also respects the cohomological grading, and $\cH_Q$ is then a $\bbZ\times\N Q_0$-graded algebra.

\subsection{Generating series}

For a cohomologically graded vector space $V = \bigoplus_{n\in\bbZ} V^n$ we consider the Poincaré series in $\bbZ(\!(q^{1/2})\!)$ defined as\footnote{The use of half powers $q^{1/2}$ is a standard convention, which originates in enumerative theories over finite fields.}
\[
  P(V,q) = \sum_{n\in\bbZ} (-q^{1/2})^n \cdot \dim V^n.
\]
When working with quivers, we will consider generating series keeping track of the additional grading over the monoid $\N Q_0$.
The Poincaré series extends to $V = \bigoplus_{\del\in\N Q_0}V_\del$ as
\[
  P(V,x,q) \colonequals \sum_{\del\in\N Q_0} P(V_\del,q) \cdot x^\del
  =
  \sum_{\del\in\N Q_0} \sum_{n\in\bbZ} (-q^{1/2})^n \cdot \dim V^n_\del \cdot x^\del,
\]
where $x^\del = \prod_{i\in Q_0} x_i^{\del_i}$ in multi-index notation; the result is a formal series in $\bbZ(\!(q^{1/2})\!)[\![x_i \mid i\in Q_0]\!]$.
The generating series of a quiver $Q$ is precisely the Poincaré series of its CoHA:
\[
  \bbA_Q(x,q) \colonequals P(\cH_Q,x,q) = \sum_{\del\in \N Q_0} P(\cH_{Q,\del},q) \cdot x^\del.
\]
Each component $\cH_{Q,\del}$ is a shift of the $\GL_\del$-equivariant cohomology of a point, and the generating series therefore has a very explicit expression
\[
  \bbA_Q(x,q)
  = \sum_{\del\in\N Q_0} \sum_{n\in\bbZ} \frac{(-q^{1/2})^{\chi_Q(\del,\del)}}{\prod_{i\in Q_0} (1-q)(1-q^2)\cdots(1-q^{\del_i})} \cdot x^\del.
\]
Some variations of the above generating series appear in the literature.
In particular, \cite{EKL20} associates to any (symmetric) quiver a series
\[
  P^Q(y,t) = \sum_{\del\in \N Q_0} \frac{(-t)^{\sum_{i,j\in Q_0} C_{ij} \del_i\del_j}\cdot y^\del}{\prod_{i\in Q_0}(1-t^2)(1-t^4)\cdots(1-t^{2\del_i})},
\]
where $(C_{ij})_{i,j\in Q_0}$ is the adjacency matrix of $Q$.
Since we wish to categorify results from \cite{EKL20}, it is worth explaining how this series is related to $\bbA_Q(x,t)$.

\begin{lem}
  There is an equality $\bbA_Q(x,q) = P^Q(q^{-1/2}x,q^{-1/2})$.
\end{lem}
\begin{proof}
  By definition of the Euler pairing, we have
  \[
    \sum_{i,j\in Q_0} C_{ij} \del_i\del_j = \sum_{i\in Q_0} \del_i^2 - \chi_Q(\del,\del) = \sum_{i\in Q_0} \del_i(\del_i+1) - \sum_{i\in Q_0} \del_i - \chi_Q(\del,\del).
  \]
  Hence, after setting $t=q^{-1/2}$ we obtain
  \[
    \begin{aligned}
      P^Q(y,q^{-1/2}) &=  \sum_{\del\in \N Q_0} \frac{(-q^{1/2})^{\chi_Q(\del,\del) + \sum_{i\in Q_0} \del_i}\cdot y^\del}{\prod_{i\in Q_0}(-q^{1/2})^{\del_i(\del_i+1)}(1-q^{-1})(1-q^{-2})\cdots(1-q^{-\del_i})} \\
                      &=  \sum_{\del\in \N Q_0} \frac{(-q^{1/2})^{\chi_Q(\del,\del)}\cdot (q^{1/2}y)^\del}{\prod_{i\in Q_0} (1-q)(1-q^2)\cdots(1-q^{\del_i})},
    \end{aligned}
  \]
  which becomes equal to $\bbA_Q(x,q)$ after setting $y = q^{-1/2}x$.
\end{proof}


\section{Unlinking}\label{sec:unlink}

In this section we relate the unlinking procedure of \cite{EKL20} to a stratification on the moduli spaces of a quiver.
Throughout, we fix a quiver $Q = (Q_0,Q_1)$ which contains a distinguished two-cycle consisting of arrows $c\colon 0\to 1$, $d\colon 1\to 0$ between distinct vertices $0,1\in Q_0$.

\subsection{Unlinking}\label{ssec:unlink}
We start by describing the unlinking process of \cite{EKL20} producing the quiver $Q^\UU$ out of $Q$.
Our description differs somewhat from the one given in \cite{EKL20}, the main difference being that we explicitly name all arrows in $Q^\UU$ and that our definition applies also to non-symmetric quivers.

\begin{defn}\label{def:unl}
  The \emph{unlinking} of $Q$ at $(c,d)$ is the quiver $Q^\UU = (Q_0^\UU,Q_1^\UU)$ with vertices
  \[
    Q_0^\UU \colonequals Q_0 \sqcup \{\star\}
  \]
  and arrows obtained from the arrows $a\colon i\to j$ in $Q_1$ via the following recipe:
  \begin{itemize}
  \item if $i \neq 0,1$ and $j\neq 0,1$ there is a single arrow $a\colon i\to j$ in $Q_1^\UU$,
  \item if $i \neq 0,1$ and $j\in \{0,1\}$ there are two arrows in $Q^\UU_1$ denoted
    \[
      a \colon i \to j,\quad a_\star \colon i\to \star,
    \]
  \item if $i \in\{0,1\}$ and $j\neq 0,1$ there are two arrows in $Q^\UU_1$ denoted
    \[
      a \colon i \to j,\quad a^\star \colon \star \to j,
    \]
  \item if $i\in\{0,1\}$ and $j\in \{0,1\}$ and additionally $a\not\in \{c,d\}$ there are four arrows in $Q^\UU_1$ denoted
    \[
      \begin{gathered}
        a \colon i \to j,\quad a^\star \colon \star \to j,\\
        a_\star \colon i \to \star,\quad a_\star^\star \colon \star \to \star,
      \end{gathered}
    \]
  \item the arrow $a=c$ contributes a single arrow $c^\star_\star\colon \star \to \star$ in $Q^\UU_1$, while the arrow $d$ contributes no arrows.
  \end{itemize}
\end{defn}

Let $(C_{ij})_{i,j\in Q_0}$ denote the adjacency matrix of $Q$.
Then by counting the contributions of each arrow in \cref{def:unl} we see that the adjacency matrix $(C^\UU_{ij})_{i,j\in Q_0^\UU}$ of $Q^\UU$ is given by
\[
  C^\UU_{ij} =
  \begin{cases}
    C_{ij} & \text{if } i,j\in Q_0 \text{ with } (i,j) \neq (0,1), (1,0) \\
    C_{ij} - 1 & \text{if } (i,j) = (0,1) \text{ or } (i,j) = (1,0)\\
    C_{i0} + C_{i1} & \text{if } i \in Q_0\setminus\{0,1\} \text{ and } j = \star \\
    C_{i0} + C_{i1}-1 & \text{if } i \in \{0,1\} \text{ and } j = \star \\
    C_{0j} + C_{1j} & \text{if } i = \star \text{ and } j \in Q_0\setminus\{0,1\} \\
    C_{0j} + C_{1j}-1 & \text{if } i = \star \text{ and } j \in \{0,1\} \\
    C_{00} + C_{01} + C_{10} + C_{11} - 1 & \text{if } i = j = \star
  \end{cases}.
\]
If $Q$ is a symmetric quiver, this agrees with the adjacency matrix constructed in \cite{EKL20}.

We define a map relating the dimension vectors of $Q$ and $Q^\UU$.
Let $u\colon \N Q_0^\UU \to \N Q_0$ be the linear map which sends $\eps$ to the dimension vector $u(\eps) = (u(\eps)_i)_{i\in Q_0}$ with
\[
  u(\eps)_i = \begin{cases}
    \eps_i + \eps_\star & \text{if } i =0,1\\
    \eps_i & \text{otherwise}.
  \end{cases}
\]
Then this linear map relates the Euler forms of the two quivers according to the following formula.

\begin{lem}
  For all $\eps,\eps'\in\N Q_0^\UU$
  \[
    \chi_{Q^\UU}(\eps,\eps') - \chi_Q(u(\eps),u(\eps')) = \eps_0\eps_1' + \eps_1\eps_0'.
  \]
\end{lem}
\begin{proof}
  The Euler forms are related to the adjacency matrices of $Q$ and $Q^\UU$ via
  \[
    \chi_Q(\del,\del') = \sum_{i\in Q_0} \del_i\del_i' - \sum_{i,j\in Q_0} C_{ij} \del_i\del_j',\quad
    \chi_{Q^\UU}(\eps,\eps') = \sum_{i\in Q_0^\UU} \eps_i\eps_i' - \sum_{i,j\in Q_0^\UU} C^\UU_{ij} \eps_i\eps_j'.
  \]
  Since $u(\eps)_i = \eps_i$ for $i\neq 0,1$, the difference in the first summations is given by
  \[
    \begin{aligned}
      \sum_{i\in Q_0^\UU} \eps_i\eps_i'- \sum_{i\in Q_0} u(\eps)_iu(\eps')_i
      &=
        \eps_\star\eps_\star' + \sum_{i=0,1} \left(\eps_i\eps_i'  -  (\eps_i + \eps_\star)(\eps_i'+\eps_\star')\right)
      \\
      &=
        -(\eps_\star\eps_\star' + \eps_0\eps_\star' + \eps_\star\eps_0' + \eps_1\eps_\star' + \eps_\star\eps_1')
    \end{aligned}
  \]
  For the second summation, the explicit relation between $C^\UU_{ij}$ and $C_{ij}$ yields
  \[
    \begin{aligned}
      \sum_{i,j\in Q_0^\UU} C^\UU_{ij}\eps_i\eps_j'
      &=
        \left(\sum_{i,j\in Q_0} C_{ij}\eps_i\eps_j'\right)
        - \eps_0\eps_1'
        - \eps_1\eps_0'.
      \\
      &\quad+
        \left(\sum_{i\in Q_0,j\in\{0,1\}} C_{ij}\eps_i\eps_\star'\right)
        - \eps_0\eps_\star'
        - \eps_1\eps_\star'    \\
      &\quad+
        \left(\sum_{i\in \{0,1\},j\in Q_0} C_{ij}\eps_\star\eps_j'\right)
        - \eps_\star\eps_0'
        - \eps_\star\eps_1' \\
      &\quad+
        \left(\sum_{i,j\in \{0,1\}} C_{ij}\eps_\star\eps_\star'\right)
        - \eps_\star\eps_\star'
      \\&=
      \sum_{i,j\in Q_0} C_{ij} u(\eps)_i u(\eps')_j
      - (\eps_\star\eps_\star'
      + \eps_0\eps_*'
      + \eps_\star\eps_0'
      + \eps_1\eps_*'
      + \eps_\star\eps_1')
      - \eps_0\eps_1'
      - \eps_1\eps_0'.
    \end{aligned}
  \]
  Putting this together, we find
  \[
    \chi_{Q^\UU}(\eps,\eps') - \chi_Q(u(\eps),u(\eps'))=  \eps_0\eps_1' + \eps_1\eps_0'. \qedhere
  \]
\end{proof}

In particular, we find that the dimensions of the moduli spaces of representations are related by
\[
  \dim \cM_{u(\eps)}(Q) = \dim \cM_\eps(Q^\UU) + 2\eps_0\eps_1
\]

\subsection{Stratification}\label{ssec:strat}

The two-cycle $(c,d)$ also gives rise to a stratification of the moduli spaces of $Q$.
Given a dimension vector $\del \in \N Q_0$, there is a well-defined $\GL_\del$-invariant function
\[
  \rk d \colon R_\del(Q) \to \N,\quad \rho \mapsto \rk \rho_d.
\]
In what follows we will denote the level sets of this function by $S_{\del,\ell} \colonequals \{\rk d=\ell\}$, and write $\overline S_{\del,\ell} \colonequals \left\{\ \rk d \leq \ell\ \right\}$ for the sublevel sets.
Since $\rk d$ is lower semicontinuous each $\overline S_{\del,\ell}$ is a closed subvariety and $S_{\del,\ell} =  S_{\del,\ell}\setminus \overline S_{\del,\ell-1}$ is locally closed.
We therefore obtain a stratification of $R_\del(Q)$ 
\[
  R_\del(Q) = \bigsqcup_{\ell\leq \min(\del_0,\del_1)} S_{\del,\ell},
\]
into locally closed subsets.
The stratification is $\GL_\del$-invariant, and therefore descends to a stratification of the moduli stack
\[
  \cM_\updelta(Q) = \bigsqcup_{\ell \leq \min(\del_0,\del_1)} \cS_{\del,\ell},
\]
where the strata $\cS_{\del,\ell} = [S_{\del,\ell}/\GL_\del]$ are locally closed substacks.
These stacky strata can be presented as a quotient of an affine space by an algebraic group.

\begin{prop}\label{prop:stratumquot}
  For each $\del\in\N Q_0$ and $0\leq \ell\leq \min(\del_0,\del_1)$, consider the subvariety of $S_{\del,\ell}$
  \[
    F_{\del,\ell}
    \colonequals 
    \left\{
      \rho \in R_\del(Q)
      \ \middle|\ 
      \rho_d = 
      \left(
        \begin{array}{c|c}
          0 & 0 \\[.2em]\hline\rule{0pt}{1em} 0 & I_{\ell\times\ell}
        \end{array}
      \right)
    \right\},
  \]
  and the subgroup of $\GL_\del$ given by
  \[
    P_{\del,\ell} \colonequals \left\{
      (g_i)_{i\in Q_0} \in \GL_\del
      \ \middle|\ 
      \begin{aligned}
      g_0 &=
      \left(
        \begin{array}{c|c}
          a & 0 \\[.2em]\hline\rule{0pt}{1em} b & c
        \end{array}
      \right)
      \\
      g_1 &= 
      \left(
        \begin{array}{c|c}
          d & e \\[.2em]\hline\rule{0pt}{1em} 0 & c
        \end{array}
      \right)     
      \end{aligned}
      ,\quad
    \begin{gathered}
      \text{ for some } c\in \GL(\ell),\\
      a\in \GL(\del_0-\ell),\,
      d\in \GL(\del_1-\ell),
      \\
      b\in\Mat(\del_0-\ell,\ell),\,
      e\in\Mat(\ell,\del_1-\ell)\\
      \end{gathered}
    \right\}.
  \]
  Then the natural inclusions yield an isomorphism $[F_{\del,\ell}/P_{\del,\ell}] \hookrightarrow \cS_{\del,\ell}$.
  In particular, the stratum $\cS_{\del,\ell}$ has codimension $\del_0\del_1-(\del_0+\del_1-\ell)\ell$ in $\cM_\del(Q)$.
\end{prop}
\begin{proof}
  For each pair $(\del,\ell)$ the stratum $S_{\del,\ell}$ is a Zariski-locally trivial fibre bundle
  \[
    h\colon S_{\del,\ell} \to \Gr(\delta_0,\ell) \times \Gr(\delta_1,\delta_1-\ell),\quad \rho \mapsto (\im \rho_d,\ \ker \rho_d).
  \]
  This map is moreover $\GL_\del$-equivariant with respect to the natural action of the factor $\GL(\delta_0)$ on subspaces of $\C^{\delta_0}$ and the transpose action of the factor $\GL(\delta_1)$ on subspaces of $\C^{\delta_1}$, with all other factors acting trivially.
  If $V\subset \C^{\delta_0}$ denotes the subspace spanned by the last $\ell$ basis vectors and $W \subset \C^{\delta_1}$ the subspace spanned by the first $\delta_1-\ell$ basis vectors, then 
  \[
    h^{-1}((V,W))
    =
    \left\{
      \rho \in R_\del(Q)
      \ \middle|\ 
      \rho_d = 
      \left(
        \begin{array}{c|c}
          0 & 0 \\[.2em]\hline\rule{0pt}{1em} 0 & g
        \end{array}
      \right)
      \text{ for some } g\in \GL(\ell)
    \right\}.
  \]
  Let $H = \mathsf{stab}((V,W))$ be the $\GL_\del$-stabiliser of $(V,W)$, which is given by $(g_i)_i\in\GL_\del$ where $g_0$ is block lower-triangular and $g_1$ is block upper-triangular.
  Since the Grassmanians are homogeneous spaces, the inclusion of $(V,W)$ descends to an isomorphism $\bB H \xrightarrow{\sim} (\Gr(\delta_0,\ell) \times \Gr(\delta_1,\delta_1-\ell))/\GL_\del$, and pulling this isomorphism back along $h$ yields an isomorphism
  \[
    \left[h^{-1}((V,W))/H\right] \xrightarrow{\ \sim\ } \cS_{\del,\ell}.
  \]
  Finally, we consider the map $h^{-1}((V,W)) \to \GL(\ell)$ mapping $\rho$ to the block-matrix $g$ appearing in $\rho_d$, which becomes $H$-equivariant for the obvious action on the block matrix.
  The fibre over the identity is $F_{\del,\ell} \subset h^{-1}((V,W))$ and the stabiliser is the subgroup $P_{\del,\ell}\subset H$.
  A similar argument then yields that the inclusions induce an isomorphism
  $\left[F_{\del,\ell}/P_{\del,\ell}\right] \xrightarrow{\sim} \left[h^{-1}((V,W))/H\right]$.
\end{proof}

\begin{rem}
  The above construction was inspired by a combinatorial argument due to Neil Strickland on MathOverflow \cite{Strickland}, which relates the count of maps between $\mathbb{F}_q$-vector spaces to q-binomials.
\end{rem}

Because the space $F_{\del,\ell}$ is contractible, the cohomology of the strata can again be expressed as a product of equivariant cohomology of a point for different groups.
In particular, we have the following.

\begin{cor}\label{cor:evencohom}
  The cohomology $\rH^\bullet(\cS_{\del,\ell})$ and homology $\rH^\BM_\bullet(\cS_{\del,\ell})$ are concentrated in even degrees.
\end{cor}

The closure ${\overline\cS}_{\del,\ell} = \left[\overline S_{\del,\ell}/\GL_\del\right]$ of each stacky stratum is a closed (possibly singular) substack of $\cM_\del(Q)$, and these substacks give a filtration
$\overline \cS_{\del,0} \subset \ldots \subset \overline \cS_{\del,\min(\del_0,\del_1)} = \cM_\del(Q)$.
At the level of Borel--Moore homology we obtain a system of maps in each cohomological degree
\begin{equation}\label{eq:BMsequence}
  \rH^\BM_\bullet({\overline\cS}_{\del,0}) \longrightarrow \rH^\BM_\bullet({\overline\cS}_{\del,1}) \longrightarrow \ldots \longrightarrow \rH^\BM_\bullet({\overline\cS}_{\del,\min(\del_0,\del_1)}) = \rH^\BM_\bullet(\cM_\del(Q)).
\end{equation}
The following result shows that this defines an ascending filtration on cohomology of $\cM_\del(Q)$.

\begin{prop}\label{prop:varphiiso}
  For all $\del \in \N Q_0$ the maps in \eqref{eq:BMsequence} are injective, yielding a filtration with subquotients
  \[
    \rH^\BM_\bullet({\overline\cS}_{\del,\ell})/\rH^\BM_\bullet({\overline\cS}_{\del,\ell-1}) \cong \rH^\BM_\bullet(\cS_{\del,\ell}),
  \]
  where the isomorphism is induced by the restriction to $\cS_{\del,\ell} \subset \overline \cS_{\del,\ell}$.
\end{prop}
\begin{proof}
  For each $\ell$ the inclusion ${\overline\cS}_{\del,\ell} \hookrightarrow {\overline\cS}_{\del,\ell+1}$ has open complement $\cS_{\del,\ell+1}$, which yields a long exact sequence in Borel--Moore homology
  \begin{equation}\label{eq:longexactZZS}
    \ldots \to \rH^\BM_\bullet({\overline\cS}_{\del,\ell-1}) \to \rH^\BM_\bullet({\overline\cS}_{\del,\ell}) \to \rH^\BM_\bullet(\cS_{\del,\ell}) \to \ldots.
  \end{equation}
  We claim that this implies that each ${\overline\cS}_{\del,\ell}$ has homology concentrated in even degrees.
  This follows by induction: for $\ell=0$ we have ${\overline\cS}_{\del,0} = \cS_{\del,0}$, for which the result is \cref{cor:evencohom}, and if the claim is true for $\ell-1$ then the exact sequence \label{eq:longexactZZ} specialises to
  \[
    0 = \rH^\BM_{2n+1}({\overline\cS}_{\del,\ell-1}) \to \rH^\BM_{2n+1}({\overline\cS}_{\del,\ell}) \to \rH^\BM_{2n+1}(\cS_{\del,\ell}) = 0,
  \]
  for every odd degree $2n+1$, where the vanishing of $\rH^\BM_{2n+1}(\cS_{\del,\ell})$ again follows by \ref{cor:evencohom}.
  Likewise, in every even degree we find a short exact sequence
  \[
    0 \to \rH^\BM_{2n}({\overline\cS}_{\del,\ell-1}) \to \rH^\BM_{2n}({\overline\cS}_{\del,\ell}) \to \rH^\BM_{2n}(\cS_{\del,\ell}) \to 0,
  \]
  which shows the result for $\ell$.
\end{proof}

\subsection{Relating stratification and unlinking}\label{sec:unlinkstrat}

We now describe a relation between the stratification and the moduli spaces of $Q^\UU$.
For convenience, we define $\overline u\colon \N Q_0^\UU \to \N Q_0 \times \N$ via
\[
  \overline u(\eps) \colonequals (u(\eps),\eps_\star).
\]
Then the relation between the stratification and the unlinked quiver is given by the following map.

\begin{prop}\label{prop:unlinkstrat}
  For each $\eps \in \N Q_0^\UU$ there exists a $\GL_\eps$-equivariant homotopy equivalence
  \begin{equation}\label{eq:unlinkhomotopy}
    \uppi \colon R_\eps(Q^\UU) \to F_{\overline u(\eps)}.
  \end{equation}
\end{prop}
\begin{proof}
  Let $\eps\in\N Q_0^\UU$ and consider a representation $\rho\in R_\eps(Q^\UU)$.
  Then we define the $Q$-representation $\uppi(\rho)$ of dimension $u(\eps)$ as having the following values on arrows $a\colon i\to j$ in $Q_0$:
  \begin{itemize}
  \item if $i,j \not\in \{0,1\}$ then $\uppi(\rho)_a = \rho_a$,
  \item if $i\neq 0,1$ and $j\in \{0,1\}$ then it is the block matrix
    \[
      \uppi(\rho)_a = \left(\begin{array}{c} \rho_a \\[.2em]\hline\rule{0pt}{1em}  \rho_{a_\star} \end{array}\right),
    \]
  \item if $i\in \{0,1\}$ and $j\neq 0,1$ then it is the block matrix
    \[
      \uppi(\rho)_a = \left(\begin{array}{c|c} \rho_a &  \rho_{a^\star} \end{array}\right),
    \]
  \item if $i,j\in\{0,1\}$ and $a\not\in\{c,d\}$ then it is the block matrix
  \[
    \uppi(\rho)_a = \left(\begin{array}{c|c} \rho_a &  \rho_{a^\star}\\[.2em]\hline\rule{0pt}{1em}  \rho_{a_\star} &  \rho_{a_\star^\star}\end{array}\right),
  \]
\item finally for the arrows $a=c,d$ we set
  \begin{equation}\label{eq:alphablockdecomp}
    \uppi(\rho)_c = \left(
      \begin{array}{c|c}
        0 & 0 \\[.2em]\hline\rule{0pt}{1em} 0 & \rho_{c^\star_\star}
      \end{array}
    \right),\quad
    \uppi(\rho)_d = \left(
      \begin{array}{c|c}
        0 & 0 \\[.2em]\hline\rule{0pt}{1em} 0 & I_{\eps_\star\times\eps_\star}
      \end{array}
    \right)
  \end{equation}
\end{itemize}
It is clear that this defines a closed embedding, and $F_{\overline u(\eps)}$ retracts on the image: any representation in $F_{\overline u(\eps)}$ retracts onto a one of the form $\pi(\rho)$ by shrinking the blocks of $c$ which are $0$ in \eqref{eq:alphablockdecomp}.
Moreover, this map is $\GL_\eps$-equivariant, with respect to its action on $F_{\overline u(\eps)}$ via the embedding $\GL_\eps \hookrightarrow P_{\overline u(\eps)}$ mapping $(h_i)_{i\in Q_0^\UU}$ to the element $(g_i)_{j\in Q_0}$ with $g_i = h_i$ for all $i\neq 0,1$ and
\[
  g_0 =
  \left(
    \begin{array}{c|c}
      h_0 & 0 \\[.2em]\hline\rule{0pt}{1em} 0 & h_\star
    \end{array}
  \right),
  \quad
  g_1 = 
  \left(
    \begin{array}{c|c}
      h_1 & 0 \\[.2em]\hline\rule{0pt}{1em} 0 & h_\star
    \end{array}
  \right).\qedhere
\]
\end{proof}

\begin{cor}\label{cor:isostratunl}
  For every $\eps\in \N Q_0^\UU$ the map $\uppi$ induces an isomorphism
  \[
    \rH^\bullet(\cM_\eps(Q^\UU))\cong \rH^\bullet(\cS_{\overline u(\eps)}).
  \]
\end{cor}
\begin{proof}
  The map $\uppi\colon R_\eps(Q^\UU) \hookrightarrow F_{\overline u(\eps)}$ is an equivariant homotopy equivalence, along the inclusion of groups $\GL_\eps \hookrightarrow P_{\overline u(\eps)}$.
  Since the quotient $P_{\overline u(\eps)}/\GL_\eps \cong \bbA^{(\eps_0+\eps_1)\eps_\star}$ is contractible, this inclusion is likewise a homotopy equivalence, which yields the isomorphism
  \[
    \rH^\bullet(\cM_\eps(Q^\UU))
    = \rH^\bullet_{\GL_\eps}(R_\eps(Q^\UU))
    \cong \rH^\bullet_{P_{\overline u(\eps)}}(F_{\overline u(\eps)})
    = \rH^\bullet(\cS_{\overline u(\eps)}).
    \qedhere
  \]
\end{proof}

Combining the above isomorphism with the filtration in \cref{prop:varphiiso}, we obtain the following.

\begin{thm}\label{thm:gradeddecomp}
  For each $\del\in \N Q_0$ the filtration induces a $\bbZ$-graded decomposition
  \[
    \rH^\BM_{\bullet-\chi_Q(\del,\del)}(\cM_\del(Q)) \cong \bigoplus_{u(\eps) = \del} \rH^\BM_{\bullet-\chi_{Q^\UU}(\eps,\eps)}(\cM_\eps(Q^\UU)).
  \]
\end{thm}
\begin{proof}
  The filtration in \cref{prop:varphiiso} yields a decomposition of the Borel--Moore homology
  \begin{equation}\label{eq:BMdecomp}
    \rH^\BM_{\bullet-\chi_Q(\del,\del)}(\cM_\del(Q)) \cong \bigoplus_{\ell \leq \min(\del_0,\del_1)} \rH^\BM_{\bullet-\chi_Q(\del,\del)}(\cS_{\del,\ell})
    = \bigoplus_{u(\eps) = \del} \rH^\BM_{\bullet-\chi_Q(u(\del),u(\del))}(\cS_{\overline u(\eps)}),
  \end{equation}
  where for the second equality we note that each pair $(\del,\ell)$ with $\ell\leq \min(\del_0,\del_1)$ can be written as $\overline u(\eps) = (u(\eps),\eps_\star)$ for the dimension vector $\eps$ with
  \[
    \eps_0 = \del_0 - \ell,\quad \eps_1 = \del_1 - \ell,\quad \eps_\star = \ell,
  \]
  and $\eps_i = \del_i$ for all other vertices.
  Since $\overline\cS_{\overline u(\eps)}$ is a smooth stack, we can dualise to obtain
  \begin{equation}\label{eq:shiftedBMtocoho}
    \rH^\BM_{\bullet-\chi_Q(u(\eps),u(\eps))}(\cS_{\overline u(\eps)}) \cong
    \rH^{-\bullet+\chi_Q(u(\eps),u(\eps))+2r_\eps}(\cS_{\overline u(\eps)}) \cong
    \rH^{-\bullet+\chi_Q(u(\eps),u(\eps))+2r_\eps}(\cM_\eps(Q^\UU)),
  \end{equation}
  where the second isomorphism follows from \cref{cor:isostratunl}.
  Using the expression for the codimension of $\cS_{\overline u(\eps)}$ from \cref{prop:stratumquot} it follows that
  \[
    \begin{aligned}
      2r_\eps
      &= -2\chi_Q(u(\eps),u(\eps)) - 2(\eps_0+\eps_\star)(\eps_1+\eps_\star) + 2(\eps_0+\eps_1+\eps_\star)\eps_\star \\
      &= -2\chi_Q(u(\eps),u(\eps)) - 2\eps_0\eps_1 \\
      &= -\chi_Q(u(\eps),u(\eps)) -\chi_{Q^\UU}(\eps,\eps).
    \end{aligned}
  \]
  In particular the shift in \eqref{eq:shiftedBMtocoho} is precisely $-\chi_{Q^\UU}(\eps,\eps)$.
  Since the moduli space $\cM_\eps(Q^\UU)$ is again smooth, we may dualise to rewrite \eqref{eq:BMdecomp} as
  \[
    \rH^\BM_{\bullet-\chi_Q(\del,\del)}(\cM_\del(Q))
    \cong
    \bigoplus_{u(\eps) = \del}
    \rH^{-\bullet-\chi_{Q^\UU}(\eps,\eps)}(\cM_\eps(Q^\UU))
    \cong
    \bigoplus_{u(\eps) = \del}
    \rH^\BM_{\bullet-\chi_{Q^\UU}(\eps,\eps)}(\cM_\eps(Q^\UU)).
    \qedhere
  \]
\end{proof}

\subsection{Filtration by ideals}\label{sec:filtideals}

We will explain how the decomposition in \cref{thm:gradeddecomp} interacts with the algebraic structure of the CoHA.
For every $p\in\N$ we consider the $\bbZ\times\N Q_0$-graded subspace
\[
  \cR_p \colonequals \bigoplus_{\del\in \N Q_0} \cR_{p,\del} \colonequals \bigoplus_{\del\in \N Q_0} \rH^\BM_{\bullet-\chi_Q(\del,\del)}(\overline \cS_{\del,\del_1-p}),
\]
where the homology of $\overline \cS_{\del,\del_0-p}$ sits inside the homology of $\cM_\del(Q)$ via the inclusion in
(\ref{eq:BMsequence}).
These subspaces define an exhaustive descending filtration 
\[
  \cH_Q = \cR_0 \supset \cR_1 \supset \ldots,
\]
which is finite in every graded component, since $\cR_{p,\del} = \rH^\BM_{\bullet-\chi_Q(\del,\del)}(\overline\cS_{\del,\del_1-p})$ vanishes when $p > \del_1$.
We claim that each $\cR_p$ is a CoHA right ideal.
To see this, observe that 
\[
  \rho\in\overline \cS_{\del,\del_1-p}
  \quad\Longleftrightarrow\quad
  \rk \rho_d \leq \del_1 - p
  \quad\Longleftrightarrow\quad
  \dim \ker \rho_d \geq p,
\]
and that the condition on the right hand side is preserved when taking an extension by an arbitrary representation on the right.
More formally, we have the following.

\begin{prop}\label{prop:CoHAideals}
  For every $p\in\bbZ$ the subspace $\cR_p$ is a graded right ideal of $\cH_Q$.
\end{prop}
\begin{proof}
  By construction $\cR_p$ is a graded subspace, so we show that it is a right ideal.
  For $\del,\del'\in\N Q_0$ consider the closed substack $\overline\cS_{\del,\del_1-p} \times \cM_{\del'}(Q) \subset \cM_{\del}(Q) \times \cM_{\del'}(Q)$ and the restriction
  \[
    \cE_{(\del,\del_1-p),\del'}  \colonequals p^{-1}(\overline\cS_{\del,\del_1-p} \times \cM_{\del'}(Q)),
  \]
  of the stack of extensions in the correspondence \eqref{eq:COHAcorr}.
  Given an extension $\tau\in \cE_{(\del,\del_1-p),\del'}$ of a pair $(\rho,\rho') \in \overline\cS_{\del,\del_1-p} \times \cM_{\del'}(Q)$ we see that the matrix of $d$ has the form
  \[
    \tau_d = 
    \left(
      \begin{array}{c|c}
        \rho_d & * \\[.2em]\hline\rule{0pt}{1em} 0 & \rho_d'
      \end{array}
    \right).
  \]
  Since $\dim \ker \rho_d \geq p$ it is clear from the above block-form that $\dim \ker \tau_d \geq p$ and therefore $\tau$ maps to $q(\tau) \in \overline\cS_{\del+\del',\del_1+\del'_1-p} \subset \cM_{\del+\del'}(Q)$.
  As a result, we find a commutative square
  \[
    \begin{tikzcd}
      \overline\cS_{\del,\del_1-p} \times \cM_{\del'}(Q)  \ar[d,"i"] & \ar[l,swap,"p'"]
      \cE_{(\del,\del_1-p),\del'} \ar[r,"q'"]  \ar[d,"i"]&
      \overline\cS_{\del+\del',\del_1+\del'_1-p}  \ar[d,"i"]\\
      \cM_{\del}(Q) \times \cM_{\del'}(Q) & \ar[l,swap,"p"]
      \cE_{\del,\del'} \ar[r,"q"] &
      \cM_{\del+\del'}(Q)
    \end{tikzcd}
  \]
  where the left square is cartesian, and the vertical maps are closed immersions.
  Now given an element $\alpha \in \rH^\BM_\bullet(\overline\cS_{\del,\del_1-p}) = \cR_{p,\del}$ and any $\beta \in \rH^\BM_\bullet(\cM_{\del'})$ we can form
  \[
    \alpha\otimes \beta \in \rH^\BM_\bullet(\overline\cS_{\del,\del_1-p})\otimes\rH^\BM_\bullet(\cM_{\del'}) \cong \rH^\BM_\bullet(\overline\cS_{\del,\del_1-p}\times \cM_{\del'}).
  \]
  Then the CoHA product of $\alpha$ and $\beta$ is given by mapping $i_*(\alpha\otimes \beta)$ to 
  \[
    q_*p^*i_*(\alpha\otimes \beta) = q_*i_*(p')^*(\alpha\otimes \beta) = i_*q'_*(p')^*(\alpha\otimes \beta),
  \]
  which lies in $\rH^\BM_\bullet(\overline\cS_{\del+\del',\del_1+\del'_1-p}) \subset \cR_p$.
  It follows that $\cR_p$ is a right ideal.
\end{proof}

Note that the above grading is the one in $\cH_Q$: if $Q$ is symmetric then $\cH_Q$ is a $\bbZ\times \N Q_0$-graded algebra and $\cR_p$ is an $\bbZ\times \N Q_0$-graded ideal,
while there is only an $\N Q_0$-grading if $Q$ is not symmetric.
We further remark that if $Q$ is a symmetric quiver, then it is well-known that the CoHA product can be twisted to become graded-commutative \cite{KS11}.
One-sided graded ideals in a graded-commutative are two-sided ideals, and each $\cR_p$ is therefore a two-sided ideal in this case.

\begin{thm}\label{thm:filtbyideals}
  Let $Q$ be a (symmetric) quiver and $Q^\UU$ the unlinking at a distinguished two-cycle.
  Then $\cH_Q$ has an exhaustive descending filtration by graded right (resp. two-sided) ideals
  \[
    \cH_Q = \cR_0 \supset \cR_1 \supset \ldots,
  \]
  with an isomorphism of $\bbZ\times \N Q_0$-graded vector spaces
  \[
    \gra_\bullet \cH_Q \colonequals \bigoplus_{p\in\N} \cR_p/\cR_{p+1} \cong \cH_{Q^\UU}.
  \]
  In particular, this makes $\cH_{Q^\UU}$ into an $\N Q_0$-graded (resp. $\bbZ\times\N Q_0$-graded) right $\cH_Q$-module.
\end{thm}
\begin{proof}
  It follows by \cref{prop:CoHAideals} that $\cR_p \subset \cH_Q$ are graded ideal,
  with respect to the $\bbZ\times\N Q_0$-grading if $Q$ is symmetric,
  or the $\N Q_0$-grading if it is not symmetric.
  It therefore suffices to prove that there is an isomorphism
  \[
    \begin{aligned}
      \bigoplus_{p\in\N} \cR_{p,\del}/\cR_{p+1,\del}
      &= \bigoplus_{p\in\N} \rH^\BM_{\bullet-\chi_Q(\del,\del)}(\overline \cS_{\del,\del_1-p})/\rH^\BM_{\bullet-\chi_Q(\del,\del)}(\overline \cS_{\del,\del_1-p-1})
      \\&\cong \bigoplus_{\ell=0}^{\min(\del_0,\del_1)} \rH^\BM_{\bullet-\chi_Q(\del,\del)}(\cS_{\del,\ell}) \\&\cong \bigoplus_{u(\eps) = \del} \rH^\BM_{\bullet-\chi_{Q^\UU}(\eps,\eps)}(\cM_\eps(Q^\UU))
    \end{aligned}
  \]
  which follows from \cref{prop:varphiiso} and \cref{cor:isostratunl}.
\end{proof}

In an analogous way one can show that $\cH_Q$ admits a filtration by the subspaces
\[
  \cL_p \colonequals \bigoplus_{\del\in \N Q_0} \rH^\BM_{\bullet-\chi_Q(\del,\del)}(\overline \cS_{\del,\del_0-p}),
\]
which are left ideals.
The proofs are analogous and left to the interested reader


\section{Linking}
\label{sec:link}

For a symmetric quiver with a distinguished pair of vertices, \cite{EKL20} constructs an adjacency matrix for a \emph{linked} quiver with an additional vertex.
In this section, we construct such a linking also for the non-symmetric case with an explicit labeling and study the CoHA.

\subsection{The linked quiver}
\label{ssec:link}

Let $Q$ be a quiver with a distinguished pair of vertices $\{0,1\} \subset Q_0$.
We will give an explicit construction of the linked quiver.

\begin{defn}\label{def:link}
  The \emph{linked quiver} $Q^\LL$ has vertices $Q_0^\LL = Q_0 \sqcup \{\bsquare\}$
and arrows $Q_1^\LL$ consisting of
\[
  \alpha\colon 0\to 1,\quad \beta\colon 1\to 0
\]
and additionally, for each $a \colon i\to j$ in $Q_1$, there are the following arrows in $Q_1^\LL$:
\begin{itemize}
\item if $i,j \not\in \{0,1\}$, then the same arrow $a\colon i\to j$ appears in $Q_1^\LL$.
\item if $i\in \{0,1\}$ and $j\not\in\{0,1\}$, then there are two arrows in $Q_1^\LL$ of the form
  \[
    \begin{aligned}
      a\colon i\to j,\quad &a^\bsquare\colon \bsquare \to j,
    \end{aligned}
  \]
\item if $i\not\in \{0,1\}$ and $j\in\{0,1\}$, then there are two arrows in $Q_1^\LL$ of the form
  \[
    \begin{aligned}
      a\colon i\to j,\quad &a_\bsquare\colon i\to\bsquare,
    \end{aligned}
  \]
\item if $i,j\in\{0,1\}$, then there are four arrows in $Q_1^\LL$ of the form
  \[
    \begin{aligned}
      a\colon i\to j,\quad &&a^\bsquare\colon \bsquare\to j,\\
      a_\bsquare\colon i\to\bsquare,\quad &&a_\bsquare^\bsquare\colon \bsquare\to\bsquare.
    \end{aligned}
  \]
\end{itemize}
\end{defn}

Hence there is exactly one additional two-cycle between $0$ and $1$, and there is an obvious inclusion $Q_1 \subset Q_1^\LL$, with all other arrows starting or ending in $\bsquare$.
If $Q$ is symmetric, its adjacency matrix is exactly the one described in \cite{EKL20}.

As before, we consider the CoHA $\cH_{Q^\LL} = \bigoplus_{\gamma \in \N Q_0^\LL} \cH_{Q^\LL,\gamma}$, which is naturally graded over $\bbZ \times \N Q_0^\LL$ and therefore defines a generating series
\[
  \bbA_{Q^\LL}(x,q) = P(\cH_{Q^\LL},x,q) \in \bbZ(\!(q^{\pm 1/2})\!)[x_i \mid i \in Q_0^\LL].
\]
It was shown in \cite{EKL20} that this generating series is related to $\bbA_Q(x,q)$ via the change of variables
\[
  \bbA_{Q^\LL}(x,q)|_{x_\bsquare=q^{-1/2}x_0x_1} = \bbA_Q(x,q).
\]
In order to categorify this relation, we therefore regrade $\cH_{Q^\LL}$ over $\bbZ \times \N Q_0$ along the map
\[
  v \colon \bbZ \times \N Q_0^\LL \to \bbZ \times \N Q_0,\quad (n,(\gamma',\gamma_\bsquare)) \mapsto (n-\gamma_\bsquare,\gamma'+\gamma_\bsquare(e_0+e_1)),
\]
putting $\cH_{Q^\LL,(\gamma',\gamma_\bsquare)}^n$ in $\bbZ\times \N Q_0$-degree $(n-\gamma_\bsquare,\gamma'+\gamma_\bsquare(e_0+e_1))$.

\subsection{Framed two-cycles}
\label{sec:framedquiv}

Instead of linking, one can also consider the quiver obtained by simply adding a two-cycle to $Q$: we obtain a quiver $Q^\TT$ with $Q^\TT_0 = Q_0$ and arrows
\[
  Q^\TT_1 = Q_1 \sqcup \{c\colon 0\to 1,\ d\colon 1\to 0\}.
\]
We will consider \emph{framed} representations for $Q^\TT$ with respect to a particular choice of framing data.
For each dimension vector $\del\in \N Q_0$ and $k\in \N$, we consider the space
\[
  R^{(k)}_\del(Q^\TT) = R_\del(Q^\TT) \times \Mat_\C(k,\del_0),
\]
of \emph{$k$-framed representations}, consisting of pairs $(\rho,f)$ of a $Q^\TT$-representation $\rho$ and a framing datum $f\colon \C^{\del_0} \to \C^k$.
We impose a further stability condition with respect to the two-cycle.

\begin{defn}
  A \emph{stably $k$-framed representation} is a $k$-framed representation in the subspace
  \[
    X_\del^{(k)} \colonequals \left\{
      \rho \in R_\del^{(k)}(Q^\TT)  \mid  \rk f = \rk(f \circ \rho_d) = k\right\}.
  \]
\end{defn}

We remark that there is an obvious map forgetting the framing, which we denote by
\[
  \forg \colon X_{\del}^{(k)} \to \cM_{\del}(Q^\TT).
\]
The stability condition is invariant under the action of the symmetry group $\GL_\del$, and we therefore obtain a well-defined moduli stack of stably $k$-framed representations
\begin{equation}\label{eq:stabframed}
  \cX_{\del}^{(k)} \colonequals \left[X_\del^{(k)}/ \GL_\del\right].
\end{equation}

We will now do something slightly unusual:
noting that $\GL_k$ acts on the framing data $f$ by base change on the target $\C^k$, we take the further quotient
\begin{equation}\label{eq:framedquot}
  \overline\cX_\del^{(k)} \colonequals \left[\cX_\del^{(k)}/\GL_k\right] = \left[X^{(k)}_\del/(\GL_\del \times \GL_k)\right],
\end{equation}
which can be alternatively presented as the quotient of $X^{(k)}_\del$ by the full symmetry group $\GL_{(\del,k)} = \GL_\del \times \GL_k$.
Although such a quotient is normally undesirable, it turns out to have a relation to the linked quiver.
To show this, we first present the stack \Cref{eq:framedquot} as the quotient of an affine space by an algebraic group.

\begin{prop}
  For each $\del\in \N Q_0$ and $k\in\N$, the stack $\overline\cX_\del^{(k)}$ is the quotient of the subspace
  \[
    F^{(k)}_\del \colonequals
    \left\{
      (\rho,f) \in X_\del^{(k)}
      \ \middle|\ %
      \rho_d = \left(\begin{array}{c|c} * & *\\\hline 0 & I_{k\times k} \end{array}\right),\ %
      f = \left(\begin{array}{c|c}\! 0 \!&\! I_{k\times k}\!\! \end{array}\right)
    \right\},
  \]
  where $I_{k\times k}$ denotes the $k\times k$ identity matrix, by the subgroup of $\GL_\del \times \GL_k$ of the form
  \[
    P^{(k)}_\del \colonequals
    \left\{(g, g') \in \GL_\del \times \GL_k \ \middle|\
      g_0 = \left(\begin{array}{c|c} * & *\\\hline 0 & g' \end{array}\right),\ %
      g_1 = \left(\begin{array}{c|c} * & *\\\hline 0 & g' \end{array}\right)
    \right\}.
  \]
\end{prop}
\begin{proof}
  As in the proof of \cref{prop:stratumquot}, we have a $\GL_\del \times \GL_k$-equivariant map
  \[
    h\colon X_\del^{(k)} \to \Gr(\del_1-k,\del_1) \times\Gr(\del_0-k,\del_0),\quad \rho \mapsto (\ker(f \circ \rho_d), \ker f),
  \]
  where $\GL_\del \times \GL_k$ acts on the Grassmannians by its action on the vector spaces $\C^{\del_0}$ and $\C^{\del_1}$.
  Writing $V \subset \C^{\del_1}$ and $W\subset \C^{\del_0}$ for the $k$-planes spanned by the last $k$ basis vectors, the stabiliser $\stab((V,W))$ consists of $(g,g') \in\GL_\del\times \GL_k$ with $g_0$ and $g_1$ block upper-triangular.
  The fibre is given by
  \[
    h^{-1}((V,W)) = \left\{(\rho,f) \in X_\del^{(k)}\ \middle|\ \rho_d = \left(\begin{array}{c|c} * & *\\\hline 0 & C \end{array}\right),\ f = \left(\begin{array}{c|c}\! 0 \!&\! D\!\! \end{array}\right) \right\},
  \]
  where $C$ and $D$ are arbitrary matrices in $\GL_k$.
  There is an equivariant map $h^{-1}((V,W))\to \GL_k \times \GL_k$ mapping $\rho$ to the matrices $(C,D)$, for which the fibre over $(I_{k\times k},I_{k\times k})$ is $F^{(k)}_\del$.
  The stabiliser of $(I_{k\times k},I_{k\times k})$ is $P^{(k)}_\del$, so we can conclude as in \cref{prop:stratumquot} that there is an isomorphism of stacks
  \[
    \overline\cX_\del^{(k)} \cong \left[h^{-1}((V,W))/\stab((V,W))\right] \cong \left[F^{(k)}_\del/P^{(k)}_\del\right].\qedhere
  \]
\end{proof}

The group $P^{(k)}_\del$ has the symmetry group of the linked quiver $Q^\LL$ as its Levi subgroup for some appropriate dimension vector.
As in \Cref{sec:unlink}, we therefore find a homotopy equivalence.

\begin{lem}\label{lem:LLtoTT}
  For each dimension vector $(\gamma',k) \in \N Q_0^\LL$, there is a homotopy equivalence
  \[
      \Psi\colon \cM_{(\gamma',k)}(Q^\LL) \to \overline\cX_{\gamma'+k(e_0+e_1)}^{(k)}
  \]
  of relative dimension $k^2+\gamma' k$.
\end{lem}
\begin{proof}
  Given $\rho \in R_{(\gamma',k)}(Q^\LL)$, we define a pair $(\Psi(\rho),f) \in X_{\gamma'+k(e_0+e_1)}^{(k)}$ of a representation, which we also denote as $\Psi(\rho)$ by abuse of notation, and a framing datum $f = \left(\begin{array}{c|c}\!\!0&I_{k\times k}\!\!\end{array}\right)$.
  The representation has the following values on the two-cycle:
  \[
    \Psi(\rho)_c =
    \left(\begin{array}{c|c} \rho_\alpha & 0\\\hline 0 & 0 \end{array}\right),\quad
    \Psi(\rho)_d =
    \left(\begin{array}{c|c} \rho_\beta & 0\\\hline 0 & I_{k\times k} \end{array}\right)
  \]
  and its value on the remaining arrows $a\in Q_1 \subset Q_1^\TT$ is given by
  \[
    \Psi(\rho)_a =
    \begin{cases}
      \rho_a &\text{if } i,j\not\in\{0,1\},\\
      \left(\begin{array}{c|c} \rho_a & \rho_{a^\bsquare} \end{array}\right) &\text{if } i\in\{0,1\}\ j\not\in\{0,1\}\\[1.5ex]
      \left(\begin{array}{c} \rho_a \\\hline \rho_{a_\bsquare} \end{array}\right) &\text{if }  i\not\in\{0,1\}\  j\in\{0,1\}\\[3ex]
      \left(\begin{array}{c|c} \rho_a & \rho_{a^\bsquare}\\\hline \rho_{a_\bsquare} & \rho_{a^\bsquare_\bsquare} \end{array}\right) &\text{if }  i,j\in\{0,1\}.
    \end{cases}
  \]
  By inspection, this defines a map $R_{(\gamma',k)}(Q^\LL) \to F^{(k)}_{\gamma'+k(e_0+e_1)}$ of relative dimension $k^2 + 2k\gamma'_0+k\gamma'_1$ which is a homotopy equivalence.
  This map is moreover $\GL_{(\gamma',k)}$-equivariant along the inclusion $\GL_{(\gamma',k)} \hookrightarrow P^{(k)}_{\gamma'+k(e_0+e_1)}$ mapping $g$ to the pair $(g',g'') \in \GL_\gamma'\times \GL_k$ with  
  \[
    g'_0 = \left(\begin{array}{c|c} g_0 & 0\\\hline 0 & g_\bsquare\end{array}\right),\quad
    g'_1 = \left(\begin{array}{c|c} g_1 & 0\\\hline 0 & g_\bsquare\end{array}\right),\quad
    g'_i = g_i \text{ for } i\neq 0,1,\quad
    g'' = g_\bsquare.
  \]
  The quotient $P^{(k)}_{\gamma'+k(e_0+e_1)}/\GL_{(\gamma',k)} \cong \bbA^{k(\gamma'_0+\gamma'_1)}$ is affine, hence contractible, and thereby induces a homotopy equivalence between the associated quotient stacks:
  \[
    \cM_{(\gamma',k)}(Q^\LL) 
    = \left[R_{(\gamma',k)}(Q^\LL)/\GL_{(\gamma',k)}\right]
    \to \left[F^{(k)}_{\gamma'+k(e_0+e_1)}/P^{(k)}_{\gamma'+k(e_0+e_1)}\right]
    \cong \overline\cX^{(k)}_{\gamma'+k(e_0+e_1)}.
  \]
  This map now has relative dimension $k^2 + \gamma'_0k$.
\end{proof}
In view of the relative dimension in the above lemma, we consider the shifted cohomology groups
\[
  \overline\cH_{Q^\TT,\del}^{(k)}
  \colonequals
  \rH^\bullet(\overline\cX_\del(Q^\LL))[-\chi_{Q^\TT}(\del,\del)-k^2],
\]
and we define $\overline\cH_{Q^\TT}^{(k)} \colonequals \bigoplus_{\del\in\N Q_0} \overline\cH_{Q^\TT,\del}^{(k)}$ with its natural $\bbZ\times \N Q_0$-grading.
These $\bbZ\times \N Q_0$-graded vector spaces are related to the CoHA of $Q^\LL$ as follows.
\begin{cor}\label{cor:LLtoTTco}
  For each $k \geq 0$, the map $\Psi$ induces an isomorphism of $\bbZ\times\N Q_0$-graded vector spaces
  \[
    \bigoplus_{\gamma'\in\N Q_0}\cH_{Q^\LL,(\gamma',k)}
    \cong
    \overline\cH_{Q^\TT}^{(k)}[k]
  \]
  where the left-hand side is regraded along the map $v\colon \bbZ\times\N Q_0^\LL \to \bbZ\times\N Q_0$.
\end{cor}
\begin{proof}
  The regrading $v$ maps a given index $(n,\gamma) = (n,(\gamma',k)) \in \bbZ\times\N Q_0^\LL$ to
  \[
    (m,\del) = (n-k,\gamma'+k(e_0+e_1)).
  \]
  The homotopy equivalence $\Psi$ from \cref{lem:LLtoTT} yields an isomorphism of graded vector spaces (without shifting or regrading) $\rH^\bullet(\cM_{(\gamma',k)}(Q^\LL)) \cong \rH^\bullet(\overline\cX_{\del}^{(k)})$.  
  To calculate the shift, we note that the dimension of $\overline \cX^{(k)}_\del$ is related to the Euler pairing of $Q^\TT$ via
  \[
    \begin{aligned}
      \dim \overline\cX^{(k)}_\del
      &= -\chi_{Q^\TT}(\del,\del) + k\del_0 - k^2
      \\&= -\chi_{Q^\TT}(\gamma'+k(e_0+e_1),\gamma'+k(e_0+e_1)) + k\gamma'_0.
    \end{aligned}
  \]
  Since the relative dimension of $\Psi$ is $k^2+k\gamma'_0$, the shift in $\overline\cH_{Q^\TT,\gamma'+k(e_0+e_1)}$ is
  \[
    -\chi_{Q^\TT}(\del,\del) - k^2 = \dim \overline\cX^{(k)}_\del - k\gamma'_0 - k^2 = \dim \cM_\gamma(Q^\LL) = -\chi_{Q^\LL}(\gamma,\gamma).
  \]
  It follows that $\Psi$ identifies
  \[
    \cH_{Q^\LL,(\gamma',k)}^n = \rH^{n-\chi_{Q^\LL}(\gamma,\gamma)}(\cM_\gamma(Q^\LL)) = \rH^{m+k-\chi_{Q^\TT}(\del,\del)-k^2}(\overline\cX^{(k)}_{\gamma'+k(e_0+e_1)}) = (\overline\cH_{Q^\TT,\del}^{(k)}[k])^m.
  \]
  Hence the isomorphism induced by $\Psi$ is $\bbZ\times\N Q_0$-graded via the regrading map $v$.  
\end{proof}
  
\subsection{Unlinking the two-cycle}
\label{sec:TU}

The quiver $Q^\TT$ has a distinguished two-cycle $(c,d)$ and therefore has an obvious choice of unlinking $Q^\TU = (Q^\TT)^\UU$ with
\[
  Q^\TU_0 = Q_0 \sqcup \{\star\},
\]
and additional arrows $a_*,a^*,a^*_*$ as described in \cref{def:unl}.
We will consider framed $Q^\TU$-representations for the following choice of framing data.
For each $\eps\in \N Q_0^\TU$ and $k\in \N$, the space of $k$-framed representations is the space
\[
  R_{\eps}^{(k)}(Q^\TU) = R_\eps(Q^\TU) \times \Mat_\C(k,\eps_\star),
\]
consisting of pairs $(\rho,f)$ with $f$ being the framing data.
We impose a stability condition, yielding the following space of \emph{stably $k$-framed representations}:
\[
  Y_\eps^{(k)} \colonequals \left\{ (\rho,f) \in R_\eps^{(k)}(Q^\TU) \ \middle|\ \rk f = k \right\}.
\]
There is again an equivariant map $\forg \colon Y^{(k)}_\eps \to R_\eps(Q^\TU)$ forgetting the framing data.
Taking the quotient by the symmetry group of $Q^\TU$ or the full symmetry group $\GL_\eps\times\GL_k$ yields two spaces
\[
  \cY^{(k)}_\eps = \left[Y_\eps^{(k)}/\GL_\eps\right],\quad
  \overline\cY^{(k)}_\eps = \left[Y_\eps^{(k)}/(\GL_\eps\times \GL_k)\right] = \left[\cY_\eps^{(k)}/\GL_k\right],
\]
which are the moduli space of stably framed representations and its quotient.

We now claim that the framing is compatible with the unlinking procedure: for any $\del\in \N Q_0$, $\ell \in \N$, and $k\in \N$, we consider the preimage of the stratification in \Cref{ssec:strat}
\[
  S_{\del,\ell}^{(k)} = \forg^{-1}(S_{\del,\ell}) = \{ (\rho,f) \in X_{\del}^{(k)} \mid \rk \rho_d = \ell \},
\]
which are locally closed subspaces giving a decomposition $X^{(k)}_\del = \bigsqcup_{u(\eps) = \del} S_{\del,\ell}^{(k)}$, where the map $u\colon \N Q_0^\TU \to \N Q_0^\TT$ is defined as in \Cref{ssec:unlink}.
The preimages are clearly invariant under the action of $\GL_\eps$ and $\GL_\eps\times \GL_k$ and therefore also give decompositions
\[
  \cX^{(k)}_\del = \bigsqcup_{u(\eps) = \del} \cS_{\del,\ell}^{(k)},\quad
  \overline\cX^{(k)}_\eps = \bigsqcup_{u(\eps) = \del} \overline\cS_{\del,\ell}^{(k)},
\]
where $\cS_{\del,\ell}^{(k)} = \left[S_{\del,\ell}^{(k)}/\GL_\del\right]$ and $\overline\cS_{\del,\ell}^{(k)} = \left[S_{\del,\ell}^{(k)}/(\GL_\del\times\GL_k)\right] = \left[\cS_{\del,\ell}^{(k)}/\GL_k\right]$.
As in \Cref{prop:stratumquot}, these strata can be presented as the quotient of an affine space.

\begin{lem}
  For each $\del\in \N Q_0$ and $\ell,k\in \N$, the strata can be represented as
  \[
    \cS_{\del,\ell}^{(k)} \cong \left[F_{\del,\ell}^{(k)}/P_{\del,\ell}\right],\quad
    \overline\cS_{\del,\ell}^{(k)} \cong \left[F_{\del,\ell}^{(k)}/(P_{\del,\ell}\times \GL_k)\right],
  \]
  where $F_{\del,\ell}^{(k)} \colonequals \forg^{-1}(F_{\del,\ell})$ is the preimage of the subspace in \Cref{prop:stratumquot}.
\end{lem}
\begin{proof}
  The proof is analogous to \Cref{prop:stratumquot}.
\end{proof}
As in \Cref{sec:unlinkstrat}, we find an equivariant homotopy equivalence relating the framed unlinked quiver with the strata in the framed moduli.
\begin{lem}\label{lem:homotTU}
  For every $\eps \in \N Q_0^\TU$ and $k\in\N$, the map $\uppi$ in \Cref{eq:unlinkhomotopy} extends to a map
  \[
    \uppi\colon Y^{(k)}_\eps \to F_{\overline u(\eps)}^{(k)},
  \]
  which is a $\GL_\eps\times\GL_k$-equivariant homotopy equivalence.
\end{lem}
\begin{proof}
  By construction, the points of $F_{\del,\ell}^{(k)}$ consist of pairs $(\rho,f) \in R_{\del,\ell}^{(k)}(Q^\TT)$ such that
  \[
    \rho_d =\left(
      \begin{array}{c|c}
        0 & 0 \\[.2em]\hline\rule{0pt}{1em} 0 & I_{\ell\times\ell}
      \end{array}
    \right),
  \]
  and $f$ satisfies the stability condition, which translates to $f$ being of the form
  \[
    f \colonequals \left(\begin{array}{c|c}\!\!f_1 & f_2\!\!\end{array}\right),
  \]
  for $f_1 \in \Mat_\C(k,\del_0-\ell)$ arbitrary and $f_2 \in \Mat_\C(k,\ell)$ a matrix of rank $k$.
  Given an $\eps\in \N Q_0^\TU$ with $\overline u(\eps) = (\del,\ell)$, we consider the map
  \[
    \uppi \colon Y^{(k)}_\eps \to R_{\overline u(\eps)}^{(k)}(Q^\TT)\quad
    (\tau,f_2)\mapsto
    (\rho,f) = \left(\uppi(\tau), \left(\begin{array}{c|c}\!\!0 & f_2\!\!\end{array}\right) \right)
  \]
  where $\uppi(\tau)$ is defined as in \Cref{eq:unlinkhomotopy}.
  The image of this map is exactly the subspace of $F_{\del,\ell}^{(k)}$ where $f_1 = 0$.
  This image is a retraction of $F_{\del,\ell}^{(k)}$, hence a homotopy equivalence.
  By inspection, the map is equivariant with respect to the map $\GL_\eps \to P_{\del,\ell}$ in \Cref{prop:unlinkstrat}.
\end{proof}
\begin{prop}\label{prop:decomplink}
  For each $\del\in \N Q_0^\TT$, there are decompositions in cohomology
  \[
    \rH^{\bullet-\chi_{Q^\TT}(\del,\del)}(\overline\cX^{(k)}_\del)
    \cong \bigoplus_{u(\eps) = \del} \rH^{\bullet-\chi_{Q^\TU}(\eps,\eps)}(\overline\cY^{(k)}_\eps).
  \]
\end{prop}
\begin{proof}
  Because $\overline \cX^{(k)}_\del$ is smooth of dimension $-\chi_{Q^\TT}(\del,\del) + k\del_0 - k^2$,
  by Poincar\'e--Verdier duality there is an isomorphism
  \begin{equation}\label{eq:decompXS}
    \rH^{\bullet-\chi_{Q^\TT}(\del,\del)}(\overline\cX^{(k)}_\del)
    \cong
    \rH^{\bullet+\chi_{Q^\TT}(\del,\del)-2k\del_0 +2k^2}(\overline\cX^{(k)}_\del)^\vee
    \cong
    \rH^\BM_{-\bullet-\chi_{Q^\TT}(\del,\del)+2k\del_0 -2k^2}(\overline\cX^{(k)}_\del),
  \end{equation}
  relating the cohomology of $\overline\cX^{(k)}_\del$ to its Borel--Moore homology.  
  We claim there is a filtration on the Borel--Moore homology via the Borel--Moore homology groups of the closures of the strata as in \Cref{eq:BMsequence}, which yields a decomposition
  \[
    \rH^\BM_{-\bullet-\chi_{Q^\TT}(\del,\del)+2k\del_0 -2k^2}(\overline\cX^{(k)}_\del)
    \cong
    \bigoplus_{\ell=0}^{\min(\del_0,\del_1)} \rH^\BM_{-\bullet-\chi_{Q^\TT}(\del,\del)+2k\del_0 -2k^2}(\cS^{(k)}_{\del,\ell}).
  \]
  The proof follows from a similar argument as in \Cref{prop:stratumquot}, using the fact that $\cS_{\del,\ell}^{(k)}$ has homology/cohomology concentrated in even degrees.
  As in \Cref{thm:gradeddecomp}, we may represent the pairs $(\del,\ell)$ as $\overline u(\eps)$ for some unique $\eps\in \N Q_0^\TU$, and the dimension of a stratum $\cS_{\overline u(\eps)}^{(k)}$ satisfies
  \[
    \begin{aligned}
      2\dim \cS_{\overline u(\eps)}^{(k)}
      &= 2\dim \overline \cX^{(k)}_{u(\eps)} - 2\eps_0\eps_1\\
      &= -\chi_{Q^\TT}(u(\eps),u(\eps)) +2k(\eps_0+\eps_\star)-2k^2 - \chi_{Q^\TU}(\eps,\eps).
    \end{aligned}
  \]
  Since the strata are smooth, we can dualise again to obtain
  \[
    \rH^\BM_{-\bullet-\chi_{Q^\TT}(u(\eps),u(\eps))+2k(\eps_0+\eps_\star) -2k^2}(\cS^{(k)}_{\overline u(\eps)})
    \cong
    \rH^{\bullet-\chi_{Q^\TU}(\eps,\eps)}(\cS^{(k)}_{\overline u(\eps)})
    \cong
    \rH^{\bullet-\chi_{Q^\TU}(\eps,\eps)}(\cY^{(k)}_\eps),
  \]
  where the last isomorphism follows from the homotopy equivalence \Cref{lem:homotTU}.
  Combining this isomorphism with the decomposition \Cref{eq:decompXS} yields the result.
\end{proof}
As in the previous section, we denote the shifted cohomology of $\overline \cY_\eps^{(k)}$ by
\[
  \overline\cH_{Q^\TU,\eps}^{(k)}
  \colonequals
  \rH^\bullet(\overline\cY_\eps^{(k)})[-\chi_{Q^\TU}(\eps,\eps)-k^2],
\]
and we define $\overline\cH_{Q^\TU}^{(k)} \colonequals \bigoplus_{\eps\in\N Q_0} \overline\cH_{Q^\TU,\eps}^{(k)}$.
As in \cref{sec:unlink}, we then obtain the following corollary.
\begin{cor}\label{cor:LTUisovs}
  For each $k \geq 0$, there is an isomorphism of $\bbZ\times\N Q_0$-graded vector spaces
  \[
    \bigoplus_{\gamma'\in\N Q_0} \cH_{Q^\LL,(\gamma',k)} \cong
    \overline\cH_{Q^\TT}^{(k)}[k]
    \cong
    \overline\cH_{Q^\TU}^{(k)}[k].
  \]
\end{cor}

\subsection{A chain differential}
\label{sec:grass}

We will assemble the $\bbZ\times\N Q_0$-graded vector spaces $\overline\cH_{Q^\TT}^{(k)} \cong \overline\cH_{Q^\TU}^{(k)}$ into a chain complex:
writing $\Ch_\bullet^{\bbZ\times\N Q_0}(\bbQ)$ for the category of chain complexes of $\bbZ\times\N Q_0$-graded vector spaces over $\bbQ$,
we will construct graded maps $\d_k\colon \overline\cH_{Q^\TU}^{(k+1)}[k+1] \to \overline\cH_{Q^\TU}^{(k)}[k]$, yielding an object $(\overline\cH_{Q^\TU},\d) \in \Ch_\bullet^{\bbZ\times\N Q_0}(\bbQ)$.
Hence, $(\overline\cH_{Q^\TU},\d)$ will have three gradings:
there are the cohomological $\bbZ$-grading and $\bbN Q_0$-grading on each $\overline\cH^{(k)}_{Q^\TU}$,
and an additional \emph{homological grading} in $k$.

We will construct the operator $\d$ by relating the spaces $\overline\cH_{Q^\TU}^{(k)}$ to the cohomology of Grassmannians.
We start by observing that for every $\eps\in\N Q_0^\TU$ and $k\geq 0$ there is an isomorphism
\begin{equation}\label{eq:Yrewrite}
  \overline\cY^{(k)}_\eps \cong \left[(Y^{(k)}_\eps/\GL_k)/\GL_\eps\right] \cong \left[(\Gr_{\eps_\star,k} \times R_\eps(Q^\TU))/\GL_\eps\right],
\end{equation}
where $\Gr_{n,m}$ for $n,m\in\N$ denotes the Grassmannian of $(n-m)$-dimensional subspaces in $\C^n$:
\[
  \Gr_{n,m} \colonequals \{M \in \Mat_\C(m,n) \mid \rk M = m \} / \GL_m.
\]
Here the group $\GL_\eps = \GL_{\eps_\star} \times \GL_{\eps'}$ acts with the usual action of $\GL_{\eps_\star}$ by base change on the Grassmannian, while the factor $\GL_{\eps'}$ for $\eps' \in \N Q_0$ acts trivially.
Using the above isomorphism, we obtain the following decomposition for the cohomology of $\overline\cY_\eps^{(k)}$.

\begin{lem}\label{lem:framed-as-equiv-tensor}
  For each $\eps\in\N Q_0^\TU$ and $k\geq 0$, there is a canonical isomorphism
  \[
    \rH^\bullet(\overline\cY^{(k)}_\eps)
    \cong
    \rH^\bullet(\left[\Gr_{\eps_\star,k}/\GL_{\eps_\star}\right]) \otimes_{\rH^\bullet(B\GL_{\eps_\star})} \rH^\bullet(\cM_\eps(Q^\TU)).
  \]
\end{lem}
\begin{proof}
  Applying the isomorphism \eqref{eq:Yrewrite}, we obtain two projections fitting into a commutative diagram
  \begin{equation}\label{eq:Y-proj-diag}
    \begin{tikzcd}
      \left[\Gr_{\eps_\star,k}/\GL_\eps\right] \ar[dr,bend right=10] &\ar[l,"g",swap] \left[(\Gr_{\eps_\star,k} \times R_\eps(Q^\TU))/\GL_\eps\right] \ar[d]\ar[r,"\forg"] & \left[R_\eps(Q^\TU)/\GL_\eps\right] = \cM_\eps(Q^\TU)\ar[dl,bend left=10]\\
      & B\GL_\eps &
    \end{tikzcd}
  \end{equation}
  Taking the cup product of the pullbacks yields the K\"unneth morphism
  \begin{equation}\label{eq:kunneth-map}
    \rH^\bullet(\left[\Gr_{\eps_\star,k}/\GL_\eps\right]) \otimes_{\rH^\bullet(B\GL_\eps)} \rH^\bullet(\cM_\eps(Q^\TU)) \to \rH^\bullet(\overline\cY_\eps^{(k)}),\quad \gamma \otimes \beta \mapsto g^*\gamma \cup \forg^*\beta,
  \end{equation}
  where the $\rH^\bullet(B\GL_\eps)$-linearity follows by commutativity of the diagram \eqref{eq:Y-proj-diag}.
  We note that \eqref{eq:kunneth-map} has to be an isomorphism: because $R_\eps(Q^\TU)$ is contractible, we have isomorphisms $\rH^\bullet(\cM_\eps(Q^\TU)) \cong \rH^\bullet(B\GL_\eps)$ and $\rH^\bullet(\overline\cY_\eps^{(k)}) \cong \left[\rH^\bullet(\Gr_{\eps_\star,k} \times \{0\}/\GL_\eps\right])$.
  Finally, decomposing $\eps = (\eps',\eps_*)$ the factor $\GL_{\eps'}$ in $\GL_\eps = \GL_{\eps'} \times \GL_{\eps_\star}$ acts trivially on the Grassmannian, so that
  \[
    \rH^\bullet(\left[\Gr_{\eps_\star,k}/\GL_\eps\right]) \cong \rH^\bullet(\left[\Gr_{\eps_\star,k}/\GL_{\eps_\star}\right]) \otimes_{\rH^\bullet(B\GL_{\eps_\star})} \rH^\bullet(B\GL_\eps),
  \]
  again via a K\"unneth isomorphism.
  Composing this isomorphism with \eqref{eq:kunneth-map} yields the result.
\end{proof}

\begin{cor}\label{cor:kunneth-shift}
  For each $\eps\in\N Q_0^\TU$ and $k\geq 0$, there is a $\bbZ$-graded isomorphism
  \[
    \overline\cH_{Q^\TU,\eps}^{(k)}[k] \cong \rH^\bullet_{\GL_{\eps_\star}}(\Gr_{\eps_\star,k})[k-k^2] \otimes_{\rH^\bullet_{\GL_{\eps_\star}}(\pt)} \cH_{Q^\TU,\eps}.
  \]
\end{cor}

Using the above isomorphism, it suffices to assemble the (shifted) cohomology groups into a chain complex of $\rH^\bullet(B\GL_n)$-modules.
For this, we use the following characterisation.

\begin{prop}[{\cite[Proposition 4.5.1]{AF23}}]\label{prop:AFbasis}
  Let $\cO_{\Gr_{n,m}}^{\oplus n} \twoheadrightarrow \cQ$ be the tautological quotient bundle.
  Then $\rH^\bullet(\left[\Gr_{n,m}/\GL_n\right]) \cong \rH^\bullet_{\GL_n}(\Gr_{n,m})$ is free as a module over $\rH^\bullet(B\GL_n) = \rH^\bullet_{\GL_n}(\pt)$ with basis
  \[
    \prod_{i=1}^m c_i^{k_i}\quad \text{ for } k_1,\ldots,k_m \in \N \text{ such that } \sum_{i=1}^{m} k_i < n-m,
  \]
  where $c_i = c_i^{\GL_n}(\cQ) \in \rH^{2i}(\left[\Gr_{n,m}/\GL_n\right])$ are the equivariant Chern classes of $\cQ$.
\end{prop}

Using the above basis, we define a chain differential as follows.

\begin{lem}\label{lem:acyccompl}
  For each $n > 0$, there are $\bbZ$-graded $\rH(B\GL_n)$-linear maps 
  \[
    \d'_m \colon \rH^\bullet(\left[\Gr_{n,m+1}/\GL_n\right])[(m+1)-(m+1)^2] \to \rH^{\bullet}(\left[\Gr_{n,m}/\GL_n\right])[m-m^2],
  \]
  such that the sum $\d' = \sum_m \d_m'$ defines an acyclic chain complex in $\Ch_\bullet(\bbQ)$
  \[
    0 \leftarrow
    \rH^\bullet(\left[\Gr_{n,0}/\GL_n\right])
    \xleftarrow{\ \d'_0\ }
    \rH^\bullet(\left[\Gr_{n,1}/\GL_n\right]) 
    \xleftarrow{\ \d'_1\ }
    \cdots
    \xleftarrow{\ \d'_{n-1}\ }
    \rH^\bullet(\left[\Gr_{n,n}/\GL_n\right])[n-n^2]
    \leftarrow 0.
  \]
\end{lem}
\begin{proof}
  Applying \cref{prop:AFbasis}, the shifted cohomology $\rH^\bullet(\left[\Gr_{n,m+1}/\GL_n\right])[(m+1)-(m+1)^2]$ is spanned by basis elements $\prod_i c_i^{k_i}$ which are homogeneous of degree $\sum_{i=1}^{m+1} 2i k_i + (m+1)^2-(m+1)$.
  We define the map $\d_m'$ on these basis elements via the rule
  \[
    \prod_{i=1}^{m+1} c_i^{k_i} \mapsto  \begin{cases}
      0 & k_{m+1} > 0 \\
    c_m \cdot \left(\prod_{i=1}^m c_i^{k_i}\right) & k_{m+1} = 0.
    \end{cases}
  \]
  One sees that the element $c_m \prod_{i=1}^m c_i^{k_i}$ has degree
  \[
    \left(\sum_{i=1}^{m} 2i k_i + 2m\right) + m^2 - m = \sum_{i=1}^{m+1} 2i k_i + (m+1)^2 - (m+1),
  \]
  in the shifted cohomology $\rH^\bullet(\left[\Gr_{n,m}/\GL_n\right])[m-m^2]$, which shows that $\d'_m$ is of degree $0$.
  Moreover, it follows from the definition that $\d_{m-1}'$ is $0$ on elements $c_m \prod_{i=1}^m c_i^{k_i}$ in the image of $\d_m'$, which shows that $\d' = \sum_m \d_m'$ is a chain differential.
  To see that it is acyclic, note that the kernel of $\d_{m-1}'$ consists exactly of those basis elements $\prod_i c_i^{k_i}$ divisible by $c_m$.
  But each such element can be written as $\prod_i c_i^{k_i} = c_m \cdot \prod_i c_i^{k_i'}$, where $k_i' = k_i$ for $i\neq m$ and $k_m' = k_m-1$, which lies in the image of $\d_m'$.
\end{proof}

Because the differential is linear over $\rH^\bullet(B\GL_n)$, it is compatible with taking tensor products.
Combining this with the previous results, we obtain the following theorem.

\begin{thm}\label{thm:differentialconstruction}
  There exist $\bbZ\times\N Q_0$-graded morphisms $\d_k \colon \overline\cH_{Q^\TU}^{(k+1)}[k+1] \to \overline\cH_{Q^\TU}^{(k)}[k]$,
  defining a $\bbZ\times\N Q_0$-graded chain complex $(\overline\cH_{Q^\TU},\d) \in  \Ch_\bullet^{\bbZ\times\N Q_0}(\bbQ)$ of the form
  \begin{equation}\label{eq:cH-chain-comp}
    0 \leftarrow \overline\cH_{Q^\TU}^{(0)} \xleftarrow{\ \d_0\ }  \overline\cH_{Q^\TU}^{(1)}[1] \xleftarrow{\ \d_1\ }  \overline\cH_{Q^\TU}^{(2)}[2] \xleftarrow{\ \d_2\ } \cdots
  \end{equation}
  with homology $\rH_\bullet(\overline\cH_{Q^\TU},\d)$ isomorphic to $\cH_Q$,
  interpreted as a $\bbZ\times\N Q_0 \times \bbZ$-graded vector space concentrated in homological degree $0$.
  In particular, $\cH_{Q^\LL}$ and $\overline\cH_{Q^\TT} \colonequals \bigoplus_{k\in\bbZ} \overline\cH_{Q^\TT}^{(k)}[k]$ obtain the structure of graded chain complexes quasi-isomorphic to $\cH_Q$:
  \[
    (\cH_{Q^\LL},\d) \cong \textstyle(\overline\cH_{Q^\TT},\d) \cong (\overline\cH_{Q^\TU},\d) \simeq \cH_Q\quad \in \Ch_\bullet^{\bbZ\times\N Q_0}(\bbQ).
  \]
\end{thm}
\begin{proof}
  Given $\eps = (\eps',\eps_\star)\in\N Q_0^\TU$, we define the restriction of $\d_k$ on $\overline\cH_{Q^\TU,\eps}^{(k+1)}[k+1]$ as $\d_k = 0$ if $\eps_\star = 0$, and for $\eps_\star > 0$ by
  \[
    \begin{aligned}
      \d_k \colon 
      \overline\cH_{Q^\TU,\eps}^{(k+1)}[k+1] \cong
      & \,\rH^\bullet(\left[\Gr_{\eps_\star,k+1}/\GL_{\eps_\star}\right])[k+1-(k+1)^2] \otimes_{\rH^\bullet(B\GL_{\eps_\star})} \cH_{Q^\TU,\eps}
      \\
      &\xrightarrow{\d_k' \otimes\id}
      \rH^\bullet(\left[\Gr_{\eps_\star,k}/\GL_{\eps_\star}\right])[k-k^2] \otimes_{\rH^\bullet(B\GL_{\eps_\star})} \cH_{Q^\TU,\eps}
      \cong
      \overline\cH_{Q^\TU,\eps}^{(k)}[k],
    \end{aligned}
  \]
  applying the isomorphism \cref{cor:kunneth-shift}.
  Because $\d' = \sum_k \d'_k$ is a $\bbZ\times\N Q_0$-graded chain differential, so is the sum $\d = \sum_k \d_k$, and we obtain a chain complex 
  \begin{equation}\label{eq:eps-chain}
    0 \leftarrow \overline\cH_{Q^\TU,\eps}^{(0)} \xleftarrow{\ \d_0\ } \overline\cH_{Q^\TU,\eps}^{(1)}[1] \xleftarrow{\ \d_1\ } \cdots \xleftarrow{\ \d_{\eps_\star-1}\ } \overline\cH_{Q^\TU,\eps}^{(\eps_\star)}[\eps_\star] \leftarrow 0,
  \end{equation}
  for each $\eps\in\N Q_0^\TU$.
  In particular, the direct sum over all $\eps$ yields the chain complex \eqref{eq:cH-chain-comp}.
  The direct summands with $\eps_* = 0$ recover a subcomplex with trivial differential
  \[
    \bigoplus_{k\geq 0} \bigoplus_{\eps'\in \N Q_0} \overline\cH_{Q^\TU,(\eps',0)}^{(k)}[k]
    = \bigoplus_{\eps'\in \N Q_0} \overline\cH_{Q^\TU,(\eps',0)}^{(0)}
    \cong \bigoplus_{\eps'\in \N Q_0} \cH_{Q,\eps'} = \cH_Q,   
  \]
  since $\overline\cH_{Q^\TU,(\eps',0)}^{(k)} = 0$ for $k > 0$.
  Hence $\cH_Q$ is a direct summand of the chain complex \eqref{eq:cH-chain-comp}, and it suffices to show that the complement is acyclic.
  To see this, we simply note that for $\eps_\star > 0$, the chain complex \eqref{eq:eps-chain} is the tensor product of the acyclic chain complex \cref{lem:acyccompl} with $\cH_{Q^\TU,\eps}$ over $\rH^\bullet(B\GL_{\eps_\star})$.
  Because $\rH^\bullet(\cM_\eps(Q^\TU)) \cong \rH^\bullet(B\GL_\eps)$ is free over $\rH^\bullet(B\GL_{\eps_\star})$, this tensor product is exact (in fact, it is just equivalent to a shift), and it follows that \eqref{eq:eps-chain} is again acyclic.
  Endowing the $\bbZ\times\N Q_0$-graded vector spaces $\cH_{Q^\LL}$ and $\overline\cH_{Q^\TT}$ with the additional homological grading
  \[
    (\cH_{Q^\LL})_k = \bigoplus_{\gamma'\in\N Q_0} \cH_{Q^\LL,(\gamma',k)},\quad
    (\overline\cH_{Q^\TT})_k = \overline\cH_{Q^\TT}^{(k)}[k],
  \]
  we can transfer the differential using the $\bbZ\times\N Q_0$-graded isomorphisms in \Cref{cor:LTUisovs}.
  Hence, it is clear that one obtains the graded chain complexes as described.  
\end{proof}

The following corollary shows that the chain complex obtained above is a categorification of the generating series identity between $Q$ and $Q^\LL$ in \cite{EKL20}.

\begin{cor}\label{cor:linkgenseries}
  For every quiver $Q$ which admits a linking $Q^\LL$, there is an identity
  \[
    \bbA_{Q^\LL}(x,q)|_{x_\bsquare = q^{-1/2}x_0x_1} = \bbA_Q(x,q^{-1/2}).
  \]
\end{cor}
\begin{proof}
  For each $\eps\in \N Q_0^\TU$ with $\eps_\star > 0$ and $n\in \bbZ$, the operator $\d$ yields a finite acyclic complex
  \[
    0 \longleftarrow \overline\cH_{Q^\TU,\eps}^{(0),n} \longleftarrow \overline\cH_{Q^\TU,\eps}^{(1),n+1} \longleftarrow \ldots \longleftarrow \overline\cH_{Q^\TU,\eps}^{(\eps_\star),n+\eps_\star} \longleftarrow 0.
  \]
  Since this complex is acyclic, the alternating sum of the dimensions of these vector spaces has to vanish.
  In particular, we find for each $\eps$ with $\eps_\star > 0$ an identity
  \[
    \sum_{k=0}^{\eps_\star} P(\overline\cH^{(k)}_{Q^\TU,\eps}[k],q) =
    \sum_{n\in\bbZ}\sum_{k=0}^{\eps_\star} (-1)^{k} (-q^{1/2})^n \dim\overline\cH_{Q^\TU,\eps}^{(k),n+k} = 0.
  \]
  The change of variables $x_\bsquare = q^{-1/2}x_0x_1$ is exactly the regrading map $v\colon \bbZ\times\N Q_0^\LL \to \bbZ\times\N Q_0$, so the isomorphisms in \Cref{cor:LTUisovs} yield an equality
  \[
    P(\cH_{Q^\LL},x,q)|_{x_\bsquare=q^{-1/2}x_0x_1}
    = \sum_{\substack{\eps\in\N Q_0^\TU\\k\in\bbN}} P(\cH_{Q^\TU,\eps}^{(k)}[k],q) \cdot x^\eps
    = \sum_{\eps\in\N Q_0} P(\cH_{Q^\TU,(\eps,0)}^{(0)},q) \cdot x^\eps
    = P(\cH_Q,x,q),
  \]
  where the final equality follows from the identity $\overline\cH_{Q^\TU,(\gamma',0)}^{(0)} = \cH_{Q,0}$.
\end{proof}

\subsection{CoHA module structure}
\label{sec:COHAlink}

We further categorify by showing that the graded chain complexes in \cref{thm:differentialconstruction} have a $\cH_Q$-linear structure compatible with the differential,
i.e. they define objects in the category $\Ch_\bullet^{\bbZ\times\N Q_0}(\cH_Q)$ of $\bbZ\times\N Q_0$-graded $\cH_Q$-modules.
To construct this, we first endow $\cH_{Q^\TU}$ with a $\cH_Q$-module structure and then extend it over the tensor product with the Grassmannian complex.
We will show that the resulting CoHA-module structure is closely related to the one on framed representations \cite{S16}.
For simplicity, we will assume that $Q$ is symmetric, so that $\cH_Q$ is a $\bbZ\times\N Q_0$-graded algebra.
The general case follows similarly by forgetting the cohomological grading.

We start by observing that for every $\del\in \N Q_0$, there is an identification $R_{(\del,0)}(Q^\TU) = R_\del(Q)$,
which yields an embedding $\cH_Q \subset \cH_{Q^\TU}$ as a $\bbZ\times\N Q_0$-graded direct summand.
This is in fact a graded subalgebra, as there are compatible identifications of the diagrams
\[
  \begin{tikzcd}
    \cM_{\del^{(1)}}(Q) \times \cM_{\del^{(2)}}(Q)\ar[d,equals] &
    \ar[l,"p"] \cE_{\del^{(1)},\del^{(2)}}(Q)\ar[d,equals] 
    \ar[r,"q"] & \cM_{\del^{(1)}+\del^{(2)}}(Q)\ar[d,equals] \\
    \cM_{(\del^{(1)},0)}(Q^\TU) \times \cM_{(\del^{(2)},0)}(Q^\TU) &
    \ar[l,"p"] \cE_{(\del^{(1)},0),(\del^{(2)},0)}(Q^\TU) 
    \ar[r,"q"] & \cM_{(\del^{(1)}+\del^{(2)},0)}(Q^\TU)
  \end{tikzcd}
\]
used to define the two CoHA multiplications.
Hence $\cH_Q \hookrightarrow \cH_{Q^\TU}$ embeds as a $\bbZ\times \N Q_0$-graded subalgebra, and in particular $\cH_{Q^\TU}$ is naturally a $\bbZ\times \N Q_0$-graded right module over $\cH_Q$.
We note that for every fixed value of $\eps_\star\geq 0$, the $\bbZ\times \N Q_0$-graded subspace
\begin{equation}\label{eq:potential-bimodule}
  \bigoplus_{\eps' \in \N Q_0} \cH_{Q^\TU,(\eps',\eps_\star)} \subset \cH_{Q^\TU},
\end{equation}
is a $\cH_Q$-submodule, since the action of $\cH_{Q,\del}$ maps $\cH_{Q^\TU,(\eps',\eps_\star)}$ to $\cH_{Q^\TU,(\eps'+\del,\eps_\star)}$.
The following lemma shows that this action is compatible with the action of $\rH^\bullet(B\GL_{\eps_\star})$.

\begin{lem}\label{lem:CoHA-GL-compat}
  For $\eps = (\eps',\eps_\star) \in \N Q_0^\TU$ and $\alpha \in \cH_{Q,\del}$, the CoHA multiplication
  \[
    (-\cdot \alpha) \colon \cH_{Q^\TU,\eps} \to \cH_{Q^\TU,(\eps'+\del,\eps_\star)}
  \]
  is a morphism of $\bbZ$-graded $\rH^\bullet(B\GL_{\eps_\star})$-modules.
  In particular, the direct sum $\bigoplus_{\eps' \in \N Q_0} \cH_{Q^\TU,(\eps',\eps_\star)}$ is a $\bbZ\times \N Q_0$-graded $\rH^\bullet(B\GL_{\eps_\star})$-$\cH_Q$-bimodule for every $\eps_\star \geq 0$.
\end{lem}
\begin{proof}
  The action of a class $\gamma \in \rH^\bullet(B\GL_{\eps_\star})$ on $\beta \in \rH^\bullet(\cM_\eps(Q^\TU))$ is given by the cup product $a_\eps^*\gamma \cup \beta$, where $a_\eps^*$ is the pullback along the composition
  \[
    a_\eps\colon \cM_\eps(Q^\TU) \to B\GL_\eps \cong B\GL_{\eps'} \times B\GL_{\eps_\star} \to B\GL_{\eps_\star},
  \]
  of the $\GL_\eps$-quotient of the structure morphism $R_\eps(Q^\TU) \to \pt$ and the projection.
  The defining correspondence for the CoHA product fits into a commutative diagram
  \[
    \begin{tikzcd}
      \cM_\eps(Q^\TU) \ar[d]&
      \cM_\eps(Q^\TU) \times \cM_{(\del,0)}(Q^\TU) \ar[d]\ar[l,"\pi_1",swap] &
      \ar[l,"p",swap] \cE_{\eps,(\del,0)}(Q^\TU)\ar[d]
      \ar[r,"q"] & \cM_{(\eps'+\del,\eps_\star)}(Q^\TU)\ar[d] \\
      B\GL_\eps \ar[drr,bend right=10] &\ar[l]
      B\GL_\eps \times B\GL_\del &
      \ar[l] B\GL_{\eps,(\del,0)} 
      \ar[r] & B\GL_{(\eps'+\del,\eps_\star)} \ar[dl,bend left=10]\\
      &&B\GL_{\eps_\star}&
    \end{tikzcd}
  \]
  where the outer composition $\cM_\eps(Q^\TU) \to B\GL_{\eps_\star}$ is $a_\eps$,
  the outer composition $\cM_{(\eps'+\del,\eps_\star)}(Q^\TU) \to B\GL_{\eps_\star}$ is $a_{(\eps'+\del,\eps_\star)}$,
  and the map $\pi_1$ is the projection.
  We recall that the CoHA product between $\alpha\in \cH_{Q,\del}$ and $\beta \in \cH_{Q^\TU,\eps}$ is given by
  \[
    \beta \cdot \alpha = q_*p^*(\pi_1^*\beta \cup \pi_2^*\alpha)
  \]
  where 
  \[
    \pi_2\colon \cM_\eps(Q^\TU) \times \cM_{(\del,0)}(Q^\TU) \to \cM_{(\del,0)}(Q^\TU) \cong \cM_\del(Q)
  \]
  denotes the second projection.
  Hence, given $\gamma\in \rH^\bullet(B\GL_{\eps_\star})$, we obtain an equality
  \[
    \begin{aligned}
      (a_\eps^*\gamma \cup \beta) \cdot \alpha
      &= q_*p^*(\pi_1^*(a_\eps^*\gamma \cup \beta) \cup \pi_2^*\alpha)\\
      &= q_*((a_\eps\circ p \circ \pi_1)^*\gamma \cup p^*(\pi_1^*\beta \cup \pi_2^*\alpha))\\
      &= q_*((a_{(\eps'+\del,\eps_\star)}\circ q)^*\gamma \cup p^*(\pi_1\beta \cup \pi_2^*\alpha))\\
      &= q_*(q^*(a_{(\eps'+\del,\eps_\star)}^*\gamma) \cup p^*(\pi_1\beta \cup \pi_2^*\alpha))\\
      &= a_{(\eps'+\del,\eps_\star)}^*\gamma \cup q_*p^*(\pi_1\beta \cup \pi_2^*\alpha)\\
      &= a_{(\eps'+\del,\eps_\star)}^*\gamma \cup (\beta \cdot \alpha),
    \end{aligned}    
  \]
  where the step from the first to the second line is the naturality of the cup product, the step from the second to the third line follows by commutativity of the diagram, and the step from the fourth to the fifth line is the projection formula for the Gysin morphism.
  The result then follows after applying the shifts in the CoHA.
\end{proof}

Because the subspaces in \eqref{eq:potential-bimodule} are $\bbZ\times \N Q_0$-graded bimodules, for every $\eps_\star \geq 0$ and $0\leq k\leq \eps_\star$
\begin{equation}\label{eq:bimodule-rewrite}
  \bigoplus_{\eps'\in\N Q_0} \overline\cH_{Q^\TU,(\eps',\eps_\star)}^{(k)}[k] \cong \rH^\bullet(\left[\Gr_{\eps_\star,k}/\GL_{\eps_\star}\right])[k-k^2]  \otimes_{\rH^\bullet(B\GL_{\eps_\star})} \bigoplus_{\eps' \in \N Q_0} \cH_{Q^\TU,(\eps',\eps_\star)}
\end{equation}
inherits the structure of a $\bbZ\times \N Q_0$-graded $\cH_Q$-module, and so does the direct sum $\overline\cH_{Q^\TU}^{(k)}$ over all $\eps_\star$.
This leads naturally to our final theorem.

\begin{thm}\label{thm:DGmodule}
  The chain differential $\d$ is $\cH_Q$-linear.
  Hence $(\cH_{Q^\LL},\d) \cong (\overline\cH_{Q^\TT},\d) \cong (\overline\cH_{Q^\TU},\d)$ are naturally objects of $\Ch_\bullet^{\bbZ\times \N Q_0}(\cH_Q)$,
  which are moreover quasi-isomorphic to the free $\bbZ\times\N Q_0$-graded $\cH_Q$-module $\cH_Q$,
  viewed as a chain complex concentrated in degree $0$.
\end{thm}
\begin{proof}
  We first note that for each $k\in\bbZ$ the components of homological degree $k$
  \[
    (\cH_{Q^\LL})_k \cong \overline\cH_{Q^\TT}^{(k)}[k] \cong \overline\cH_{Q^\TU}^{(k)}[k]
  \]
  obtain a $\bbZ\times\N Q_0$-graded $\cH_Q$-module structure via \eqref{eq:bimodule-rewrite} and \cref{lem:CoHA-GL-compat}.
  Unravelling the definition of the module structure in \cref{lem:CoHA-GL-compat},
  we see that $\alpha \in \cH_Q$ acts on elements
  \[
    \zeta \otimes \beta \in \rH^\bullet(\left[\Gr_{\eps_\star,k}/\GL_{\eps_\star}\right])[k-k^2] \otimes_{\rH^\bullet(B\GL_{\eps_\star})} \bigoplus_{\eps' \in \N Q_0} \cH_{Q^\TU,(\eps',\eps_\star)} \cong \bigoplus_{\eps'\in\N Q_0} \overline\cH_{Q^\TU,(\eps',\eps_\star)}^{(k)}[k]
  \]
  by its action on the $\cH_{Q^\TU}$-fact via $(\zeta\otimes \beta) \cdot \alpha = \zeta \otimes (\beta\cdot \alpha)$.
  Conversely, the maps $\d_k$ defined in \Cref{thm:differentialconstruction},
  act on the Grassmannian factor via $\d_k(\zeta\otimes \beta) = (\d_k'\zeta) \otimes \beta$.
  Hence, we have
  \[
    \d_k(\zeta\otimes \beta) \cdot \alpha = (\d_k'\zeta \otimes \beta)\cdot \alpha = \d_k'\zeta \otimes (\beta\cdot \alpha) = \d_k(\zeta \otimes (\beta\cdot\alpha)) = \d_k((\zeta\otimes \beta) \cdot \alpha),
  \]
  which shows that $\d_k$ and the $\cH_Q$-module action commute.
  It follows that $(\overline\cH_{Q^\TU},\d)$ is a chain complex of $\bbZ\times\N Q_0$-graded $\cH_Q$-modules,
  and the same then holds for $(\cH_{Q^\LL},\d)$, $(\overline\cH_{Q^\TT},\d)$.
  Finally, we note that the terms with $k=0$ and $\eps_\star = 0$ still form a direct summand
  \[
    \bigoplus_{\eps'\in \N Q_0} \overline\cH_{Q^\TU,(\eps',0)}^{(0)} \cong \bigoplus_{\eps'\in \N Q_0} \cH_{Q^\TU,(\eps',0)} \cong \bigoplus_{\eps'\in \N Q_0}\cH_{Q,\eps'} = \cH_Q,
  \]
  endowed with the usual $\cH_Q$-module structure.
  Because the complement is acyclic,
  it follows that $(\overline\cH_{Q^\TU},\d)$ is quasi-isomorphic to $\cH_Q$ as a chain complex of $\bbZ\times\N Q_0$-graded $\cH_Q$-modules.
\end{proof}

Finally, we address the relation with the CoHA-module structure on framed representations, which is defined in \cite{S16} as follows.
For any $\eps = (\eps',\eps_\star) \in \N Q_0^\TU$ and $\del\in \N Q_0$, we can consider the preimage
\[
  E_{\eps,\del}^{(k)} \colonequals \forg^{-1}(E_{\eps,(\del,0)}) \subset Y^{(k)}_{(\eps'+\del,\eps_\star)}
\]
of the space of extensions along the forgetful map.
We then consider the correspondence
\begin{equation}\label{eq:framedmodulecor}
  Y^{(k)}_\eps \times R_\del(Q) \xleftarrow{\ p_k\ } E_{\eps,\del}^{(k)} \xrightarrow{\ q_k\ } Y^{(k)}_{(\eps'+\del,\eps_\star)}
\end{equation}
where $q_k$ is the inclusion as a closed subvariety and $p_k$ is the vector bundle whose fibres over points $((\rho,f),\rho')$ consist of framed representations $(\tau,f)$ with $\tau$ given by an extension
\[
  \tau_a = \begin{cases}
    \left(\begin{array}{c|c} \rho_a & *\\\hline 0 & \rho'_a \end{array}\right) & \text{ for } a\in Q_1 \subset Q_1^\TU \\
    \quad\rho_a & \text{ for } a\in Q_1^\TU\setminus Q_1.
  \end{cases}
\]
Writing $\cE_{\eps,\del}^{(k)} = \left[E_{\eps,\del}^{(k)}/ \GL_{\eps,(\del,0)}\right]$ for the quotient by the parabolic subgroup, we again obtain maps on cohomology
\begin{equation}\label{eq:framedaction}
  \rH^\bullet(\cY^{(k)}_\eps \times \cM_\del(Q))
  \xrightarrow{\ p_k^*\ }
  \rH^\bullet(\cE_{\eps,\del}^{(k)})
  \xrightarrow{\ q_{k,*}\ }
  \rH^\bullet(\cY^{(k)}_{(\eps'+\del,\eps_\star)})[c],
\end{equation}
where $c = 2\dim \cY^{(k)}_{(\eps'+\del,\eps_\star)} - 2\dim \cE^{(k)}_{\eps,\del} = -2\chi_{Q^\TU}(\eps,(\del,0))$.
The sum $\cH_{Q^\TU}^{(k)} = \bigoplus_{\eps}\rH^\bullet(\cY^{(k)}_\eps)[-\chi_{Q^\TU}(\eps,\eps)]$ then obtains the module action of the form
\begin{equation}\label{eq:noquotaction}
  \cdot_{fr} \colon \cH_{Q^\TU}^{(k)} \otimes \cH_Q \to\cH_{Q^\TU}^{(k)} \quad
  \theta \otimes \alpha \mapsto q_{k,*}p_k^*(\pi_{1,k}^*\theta \cup \pi_{2,k}^*\alpha),
\end{equation}
by composing with the K\"unneth isomorphism $\rH^\bullet(\cY^{(k)}_\eps) \otimes \rH^\bullet(\cM_\del(Q))\to  \rH^\bullet(\cY^{(k)}_\eps \times \cM_\del(Q))$ defined by the projections 
$\pi_{1,k} \colon \cY^{(k)}_\eps \times \cM_\del(Q) \to \cY^{(k)}_\eps$, $\pi_{2,k} \colon \cY^{(k)}_\eps \times \cM_\del(Q) \to \cM_\del(Q)$.

\begin{prop}[{\cite[Proposition 4.1.1]{S16}}]
  For a symmetric quiver $Q$, the map \eqref{eq:noquotaction} makes $\cH_{Q^\TU}^{(k)}$ into a $\bbZ\times\N Q_0$-graded right $\cH_Q$-module.  
\end{prop}

\begin{rem}
  There is a slight difference between our framing data and the one used in \cite{S16}.
  They use a framing of the form $\C^k \to \C^{\del_i}$, whereas we have chosen a framing datum $\C^{\del_i} \to \C^k$ in the opposite direction.
  As a result, they find a left-module action, where we obtain a right-module action.
\end{rem}

We now note that all the maps in \eqref{eq:framedmodulecor} are $\GL_k$-equivariant, so one can instead take the quotient by the additional group $\GL_k$.
This yields maps in cohomology
\begin{equation}\label{eq:framedredpq}
  \begin{tikzcd}
    \rH^\bullet(\overline\cY^{(k)}_\eps \times \cM_\del(Q))
    \ar[r,"p_k^*"]
    &
    \rH^\bullet(\overline\cE_{\eps,\del}^{(k)})
    \ar[r,"q_{k,*}"]
    &
    \rH^\bullet(\overline\cY^{(k)}_{(\eps'+\del,\eps_\star)})
  \end{tikzcd}
\end{equation}
where we write $\overline\cE_{\eps,\del}^{(k)} = E_{\eps,\del}^{(k)}/(\GL_\eps\times \GL_k)$.
This gives an alternative construction of a $\cH_Q$-module action
\[
  \overline\cdot_{fr} = q_{k,*}p_k^*(\pi_1^*(-)\cup \pi_2^*(-)) \colon \overline\cH_{Q^\TU,(\eps',\eps_\star)}^{(k)} \otimes \cH_{Q,\del} \to \overline\cH_{Q^\TU,(\eps'+\del,\eps_\star)}^{(k)}.
\]
However, this action does not ``touch'' the framing data and therefore turns out to be the same action as before.

\begin{prop}\label{prop:modulesLTU}
  For each $\theta \in \overline\cH_{Q^\TU,(\eps',\eps_\star)}^{(k)}$ and $\alpha \in \cH_{Q,\del}$, there is an equality
  \[
    \theta \cdot \alpha = \theta \mathbin{\overline\cdot_{fr}} \alpha.
  \]
  In particular, the quotient by $\GL_k$ induces a module map $\overline\cH_Q^{(k)} \to \cH_Q^{(k)}$.
\end{prop}
\begin{proof}
  For each $\eps'\in\N Q_0$ and $\eps_\star\geq 0$, the isomorphism
  \[
    \rH^\bullet(\left[\Gr_{\eps_\star,k}/\GL_{\eps_\star}\right]) \otimes_{\rH^\bullet(B\GL_{\eps_\star})} \cH_{Q^\TU,(\eps',\eps_\star)} \longrightarrow \overline\cH_{Q^\TU,(\eps',\eps_\star)}^{(k)}
  \]
  can be presented (up to a shift) as $\zeta \otimes \beta \mapsto \rho^*\zeta \cup \forg^*\beta$, where $\rho$ is the composition of projections
  \[
      \rho\colon \left[\Gr_{\eps_\star,k} \times R_\eps(Q^\TU)/\GL_\eps\right]
      \to \left[\Gr_{\eps_\star,k}/\GL_\eps\right]
      \cong  \left[\Gr_{\eps_\star,k}/\GL_\eps\right] \times B\GL_{\eps'}
      \to
      \left[\Gr_{\eps_\star,k}/\GL_{\eps_\star}\right]
  \]
  and $\forg \colon \left[\Gr_{\eps_\star,k} \times R_\eps(Q^\TU)/\GL_\eps\right] \to \cM_\eps(Q^\TU)$ is the projection.
  Hence, identifying $\zeta\otimes\beta$ with $\theta = \rho^*\zeta \cup \forg^*\beta$, we have for $\alpha \in \cH_{Q,\del}$
  \[
    \begin{aligned}
      \theta\cdot\alpha &= \rho^*\zeta \cup \forg^*(\beta \cdot \alpha) 
      \\&= \rho^*\zeta \cup \forg^*q_*p^*(\pi_1^*\beta \cup \pi_2^*\alpha)
      \\&= \rho^*\zeta \cup \forg^*q_*p^*(\pi_1^*\beta \cup \pi_2^*\alpha)
      \\&= \rho^*\zeta \cup \forg^*q_*p^*(\pi_1^*\beta \cup \pi_2^*\alpha) 
    \end{aligned}
  \]
  We then note that there is a commutative diagram
  \[
    \begin{tikzcd}
      \ar[d,"{(\forg,\id)}"] \overline\cY_\eps^{(k)} \times \cM_\del(Q) &\ar[l,"p_k"] \overline\cE_{\eps,\del}^{(k)}\ar[d,"\forg"] \ar[r,"q_k"] & \overline\cY^{(k)}_{(\eps'+\del,\eps_\star)} \ar[d,"\forg"]\\
      \cM_\eps(Q^\TU) \times \cM_\del(Q) & \ar[l,"p"] \cE_{\eps,(\del,0)} \ar[r,"q"] & \cM_\eps(Q^\TU)
    \end{tikzcd}
  \]
  where the right square is Cartesian.
  This implies that $\forg^*q_* = q_{k,*}\forg^*$ and $\forg^*p^* = p_k^*(\forg,\id)^*$.
  Similarly, there is a commutative diagram
  \[
    \begin{tikzcd}
      \ar[dr,"\rho \circ \pi_{1,k}"] \overline\cY_\eps^{(k)} \times \cM_\del(Q) &\ar[l,"p_k",swap] \overline\cE_{\eps,\del}^{(k)}\ar[r,"q_k"] & \overline\cY^{(k)}_{(\eps'+\del,\eps_\star)} \ar[dl,"\rho",swap]\\
      & \left[\Gr_{\eps_\star,k}/\GL_{\eps_\star}\right] &
    \end{tikzcd}
  \]
  which implies that $q_k^*\rho^* = p_k^*\pi_{1,k}^*\rho^*$.
  Putting these two identities together and applying the projection formula yields
  \[
    \begin{aligned}
      \rho^*\zeta \cup \forg^*q_*p^*(\pi_1^*\beta \cup \pi_2^*\alpha)
      &= \rho^*\zeta \cup q_{k,*}\forg^*p^*(\pi_1^*\beta \cup \pi_2^*\alpha)
      \\&= q_{k,*}(q_k^*\rho^*\zeta \cup \forg^*p^*(\pi_1^*\beta \cup \pi_2^*\alpha))
      \\&= q_{k,*}(p_k^*\pi_{1,k}^*\rho^*\zeta \cup p_k^*((\forg,\id)^*\pi_1^*\beta \cup (\forg,\id)^*\pi_2^*\alpha))
      \\&= q_{k,*}p_k^*(\pi_{1,k}^*\rho^*\zeta \cup \pi_{1,k}^*\forg^*\beta \cup \pi_{2,k}^*\alpha)
      \\&= q_{k,*}p_k^*(\pi_{1,k}^*\theta \cup \pi_{2,k}^*\alpha)
      \\&= \theta \mathbin{\overline\cdot_{fr}} \alpha.
    \end{aligned}
  \]
  Hence $\cdot = \overline\cdot_{fr}$ as claimed.
  For the final statement, note that the quotients by $\GL_k$ fit into a commutative diagram
  \[
    \begin{tikzcd}
      \cY^{(k)}_\eps \times \cM_\del(Q) \ar[d]& \ar[l,"p_k",swap] \cE^{(k)}_{\eps,\del} \ar[r,"q_k"] \ar[d]& \cY^{(k)}_{(\eps'+\del,\eps_\star)} \ar[d]\\
      \overline\cY^{(k)}_\eps \times \cM_\del(Q) & \ar[l,"p_k",swap] \overline\cE^{(k)}_{\eps,\del} \ar[r,"q_k"] & \overline\cY^{(k)}_{(\eps'+\del,\eps_\star)}
    \end{tikzcd}
  \]
  where the right square is Cartesian.
  Therefore the pullback maps $\rH^\bullet(\overline\cY^{(k)}_\eps) \to \rH^\bullet(\cY^{(k)}_\eps)$ are compatible with the $\cH_Q$-action obtained from the two diagrams, so that the induced map $\overline\cH_{Q^\TU}^{(k)} \to \cH_{Q^\TU}^{(k)}$ is a morphism of $\cH_Q$-modules.
\end{proof}


\printbibliography%
\end{document}